\documentclass[11pt,reqno]{amsart}
\usepackage{latexsym,amsmath}
\usepackage[left=3cm,top=2.5cm,right=3cm,bottom=2.5cm]{geometry}
\usepackage{color}
\usepackage{amsthm}
\usepackage{amssymb}
\usepackage{epsfig}
\usepackage{enumerate} 
\usepackage{dsfont}

\usepackage{datetime}
\usepackage{hyperref}

\let\ykwhat\widehat

\def\RR{{\mathbb R}}
\def\QQ{{\mathbb Q}}

\def\NN{{\mathbb R}}

\let\epsilon\varepsilon

\def\({\left(}
\def\){\right)}  
\newcommand{\abs}[1]{\left|{#1}\right|}
\newcommand{\prob}[1]{\ensuremath{\mathbf{P}\left(\,#1\,\right)}}

\def\llfloor{\left\lfloor}
\def\rrfloor{\right\rfloor}
\def\llceil{\left\lceil}
\def\rrceil{\right\rceil}
\newcommand{\etal}{\textsl{et al.}}

\DeclareMathOperator{\homg}{hom}

\begin{document}

\newtheorem{theorem}{Theorem}[section]
\newtheorem{cor}[theorem]{Corollary}
\newtheorem{lemma}[theorem]{Lemma}
\newtheorem{fact}[theorem]{Fact}
\newtheorem{property}[theorem]{Property}
\newtheorem{proposition}[theorem]{Proposition}
\newtheorem{claim}[theorem]{Claim}
\newtheorem{definition}[theorem]{Definition}
\theoremstyle{definition}
\newtheorem{example}[theorem]{Example}
\newtheorem{remark}[theorem]{Remark}
\newcommand\eps{\varepsilon}
\newcommand\la {\lambda}
\newcommand{\E}{\mathbb E}
\newcommand{\Var}{{\rm Var}}
\newcommand{\Prob}{\mathbb{P}}
\newcommand{\N}{{\mathbb N}}
\newcommand{\eqn}[1]{(\ref{#1})}

\def\llfloor{\left\lfloor}
\def\rrfloor{\right\rfloor}
\def\llceil{\left\lceil}
\def\rrceil{\right\rceil}

%Balazs macros
\newcommand{\rbabs}[1]{\left|{#1}\right|}
\def \rb_ind{\mathds{1}}
\def \toinp {\buildrel {\text{d}}\over{\longrightarrow}}
\def \toinq {\buildrel {{\rm d}}\over{\longrightarrow}}
\def \toint {\buildrel {{\rm t}}\over{\longrightarrow}}
\def \toinsq {\buildrel {\square}\over{\longrightarrow}}

\newcommand{\rbprob}[1]{\ensuremath{\mathbf{P}\left(\,#1\,\right)}}
\newcommand{\rbexpect}[1]{\ensuremath{\mathbf{E}\left(\,#1\,\right)}}
\newcommand{\rbvar}[1]{\ensuremath{\mathbf{D}^2\left(#1\right)}}
\newcommand{\rbcondprob}[2]{\ensuremath{\mathbf{P}\left(\,#1\,\big|\,#2\,\right)}}
\newcommand{\rbcondexpect}[2]{\ensuremath{\mathbf{E}\left(\,#1\,\big|\,#2\,\right)}}

\title{Limits of permutation sequences}

\author[C. Hoppen]{Carlos Hoppen}
\address{Instituto de Matem\'atica, UFRGS -- Avenida Bento Gon\c{c}alves, 9500, 91509-900 Porto Alegre, RS, Brazil }
\email{\tt choppen@ime.usp.br}
\author[Y. Kohayakawa]{Yoshiharu Kohayakawa}
\address{Instituto de Matem\'atica e Estat\'\i stica, USP -- Rua do Mat\~ao
  1010, 05508--090 S\~ao Paulo, SP, Brazil} \email{\tt yoshi@ime.usp.br}
\author[C. G. Moreira]{Carlos Gustavo Moreira}
\address{IMPA -- Estrada Dona Castorina 110, 22460--320 Rio de
  Janeiro, RJ, Brazil} \email{\tt gugu@impa.br}
\author[B. R\'{a}th]{Bal\'{a}zs R\'{a}th}  
\address{Department of Mathematics, University of British Columbia -- 1984 Mathematics Road, V6T1Z2
Vancouver, BC, Canada
} \email{\tt rathb@math.ubc.ca}  
\author[R. M. Sampaio]{Rudini Menezes Sampaio}
\address{Departamento de Computa\c c\~ao, Centro de Ci\^encias, UFC --
  Campus do Pici, Bloco 910, 60451--760 Fortaleza, CE, Brazil} \email{\tt rudini@lia.ufc.br}

\thanks{The statements of some of the results of this paper appeared in the Proc. of the $21^{\textrm{st}}$ ACM-SIAM Symposium on Discrete Mathematics, SODA~2010.}

\thanks{The first author acknowledges the support by FAPERGS (Proc.~11/1436-1, 10/0388-2), FAPESP  (Proc.~2007/56496-3), and CNPq  (Proc.~484154/2010-9).  The second author was partially supported by  CNPq (Proc.~308509/2007-2, 484154/2010-9). The third author thanks CNPq for its support. The fourth author was partially supported by the OTKA (Hungarian National Research Fund) grant K 60708. The fifth author was partially supported by Funcap (Proc.~07.013.00/09) and CNPq  (Proc.~484154/2010-9).}

\thanks{The authors are grateful to NUMEC/USP, N\'{u}cleo de Modelagem Estoc\'{a}stica e
Complexidade of the University of S\~{a}o Paulo, and Project MaCLinC/USP, for supporting this research.}

%\date{\today, \currenttime}

\begin{abstract}
A permutation sequence $(\sigma_n)_{n \in \mathbb{N}}$
is said to be convergent if, for every fixed permutation $\tau$, the density of occurrences of $\tau$ in the elements of the sequence converges. 
We prove that such a convergent sequence has a natural limit object, 
namely a Lebesgue measurable function $Z:[0,1]^2 \to [0,1]$ with the additional
 properties that, for every fixed $x \in [0,1]$, the restriction $Z(x,\cdot)$ is 
a cumulative distribution function and, for every $y \in [0,1]$, the restriction
 $Z(\cdot,y)$ satisfies a ``mass'' condition. This limit process is well-behaved: 
every function in the class of limit objects is a limit of some permutation sequence,
 and two of these functions are limits of the same sequence if and only if they are equal
 almost everywhere. An ingredient in the proofs is a new model of random permutations,
 which generalizes previous models and might be interesting for its own sake. 
\end{abstract}

\maketitle
\thispagestyle{empty}

\section{Introduction}

As usual, a \emph{permutation} of a finite set $X$ is a bijective function of $X$ into itself. We shall focus on permutations $\sigma$ on the set $X=\{1,\ldots,n\}=[n]$, where $n$ is a positive integer, called the \emph{length} of $\sigma$, and is denoted by $|\sigma|$. In this work, a permutation $\sigma$ on $[n]$ is represented by $\sigma=(\sigma(1),\ldots,\sigma(n))$, and the set of all permutations on $[n]$ is denoted by $S_n$. We denote by $\mathcal{S}=\bigcup_{i=1}^\infty S_n$  the set of all finite permutations.
A \emph{graph} $G=(V,E)$ is given by its \emph{vertex set} $V$ and its \emph{edge set} $E \subseteq \big\{\{u,v\} \subset V ~:~ u \neq v\big\}$.

The main goal of this paper is to introduce a notion of convergence of a permutation sequence $(\sigma_n)_{n \in \mathbb{N}}$ and to 
identify a natural limit object for such a convergent 
sequence whose associated sequence of lengths $(|\sigma_n|)_{n \in \mathbb{N}}$ tends 
to infinity. Lov\'{a}sz and Szegedy~\cite{lovasz06} were concerned with these questions 
in the case of graph sequences $(G_n)_{n \in \mathbb{N}}$. This has been further investigated by Borgs \etal~in~\cite{borgs06b} and~\cite{borgs06c}, where, among other things, limits of graph sequences were used to characterize the testability of graph parameters. The convergence of sequences of combinatorial objects has also been addressed in other structures.
 For instance, graphs with degrees bounded by a constant have been addressed in the recent works of
 Benjamini and Schramm~\cite{benjamini_schramm} and of Elek~(\cite{elek1}, \cite{elek}).
 Elek and Szegedy~\cite{elek_szegedy} studied this problem for hypergraphs.
 See Lov\'{a}sz~\cite{lovasz} for a comprehensive survey of this area. 

Currently, the main application of our results in this paper is in property testing of permutations. Roughly speaking, the objective of testing is to decide whether a combinatorial structure satisfies some property, or to estimate the value of some numerical function associated with this combinatorial structure, by considering only a randomly chosen substructure of sufficiently large, but constant size. These problems are called \emph{property testing} and \emph{parameter testing}, respectively; a property or parameter is said to be \emph{testable} if it can be estimated accurately in this way. The algorithmic appeal of testability is evident, as, conditional on sampling, this leads to reliable constant-time randomized estimators for the said properties or parameters.
 In~\cite{rudini08d} four of the present authors address these questions through the prism of subpermutations. Among their main results are a permutation result in the direction of Alon and Shapira's~\cite{alon_shapira3} work on the testability of hereditary graph properties, and a permutation counterpart of the characterization of testable parameters by Borgs \etal~\cite{borgsstoc}.

Given the similarity of our results with the ones obtained in~\cite{lovasz06}, we briefly describe that work. 
Central in the arguments is the notion of a \emph{homomorphism} of a
 graph $F$ into a graph $G$, a function $\phi:V(F) \rightarrow V(G)$ that 
maps the vertex set $V(F)$ of $F$ into the vertex set $V(G)$ of $G$ with the property that, 
for every edge $\{u,v\}$ in $F$, the pair $\{\phi(u),\phi(v)\}$ is an edge in $G$.
 The number of homomorphisms of $F$ into $G$ is denoted by $\homg(F,G)$, while the 
\emph{homomorphism density} of $F$ into $G$ is given by the probability that a uniformly chosen $\phi$ is a homomorphism:
\begin{equation}
t(F,G)=\frac{\homg(F,G)}{|V(G)|^{|V(F)|}}.
\end{equation}

It is natural to measure the similarity between two graphs $G$ and $G'$ by comparing the homomorphism density of different graphs $F$ into them. This suggests defining a graph sequence $(G_n)_{n \in \mathbb{N}}$ as being \emph{convergent} if,
 for every (simple) graph $F$, the sequence of real numbers $(t(F,G_n))_{n \in \mathbb{N}}$ converges. 
Lov\'{a}sz and Szegedy identify a class of natural limit objects for such convergent sequences, which they call
 \emph{graphons}, in the form of symmetric Lebesgue measurable functions $W:[0,1]^2 \rightarrow [0,1]$. 
One might heuristically imagine
the adjacency matrix of a simple graph as a black-and-white television screen (a white pixel at position $(i,j)$ represents an edge between vertices $i$ and $j$), a graph
sequence as a sequence of TV sets with higher and higher resolution, and the limiting graphon $W$ as the ``perfect TV'' where each point $(x,y) \in [0,1]^2$ is a ``pixel of infinitesimal
size'' and the color of each infinitesimal pixel can be any shade of grey between black and white.

An important feature of this limit object is that it may be used to generate random graphs:
 given a graphon $W:[0,1]^2 \rightarrow [0,1]$ and a positive integer $n$, a \emph{$W$-random}
 graph $G(n,W)$ with vertex set $[n]$ is generated as follows. First, $n$ real numbers $X_1, \ldots, X_n$ are generated
 independently according to the uniform probability distribution on
 the interval $[0,1]$. Then, for every pair of distinct vertices $i$
 and $j$ in $[n]$, the pair $\{i,j\}$ is added to the edge set of the
 graph independently with probability $W(X_i,X_j)$. Heuristically, the
 adjacency matrix of $G(n,W)$ looks like the graphon $W$ if $1 \ll n$:
 we approximate the perfect TV screen by choosing the coordinates of infinitesimal pixels of $W$ at random and declaring a pixel of $G(n,W)$ white with a probability proportional to the greyness of the corresponding 
infinitesimal pixel of $W$. It is important to point out that this model of random graphs generalizes the random graph model $G(n,H)$ (see Lov\'{a}sz and S\'{o}s~\cite{lovasz_sos08}), further generalizing the classical model $G_{n,p}$ due to Erd\H{o}s-R\'{e}nyi~\cite{erdos_renyi} and Gilbert~\cite{gilbert}. 

With this, Lov\'{a}sz and Szegedy define the \emph{homomorphism density} of a $k$-vertex graph $F$ into a graphon $W$ as the probability $t(F,W)$ that $F$ is a subgraph of the $W$-random graph $G(k,W)$, which can be calculated as follows given the above definition of $G(k,W)$:
\begin{equation}\label{explicit_graphon_hom_dens_formula}
  t(F,W)= \int_0^1 \dots \int_0^1 \prod_{ \{i,j\} \in E(F)} W(x_i,x_j) \, {d} x_1 \dots {d} x_k. 
\end{equation}

 They use this to prove that, with each convergent graph sequence $(G_n)_{n \in \mathbb{N}}$, one may
 associate a graphon $W$ such that, for every fixed $F$, 
$$\lim_{n \to \infty} t(F,G_n)=t(F,W).$$
The graphon $W$ is said to be a \emph{limit} to $(G_n)_{n \in \mathbb{N}}$.
 Conversely, Lov\'{a}sz and Szegedy show that, for any fixed graphon $W$, the randomized sequence 
$(G(n,W))_{n \in \mathbb{N}}$ converges to $W$ with probability one. Hence, given a graphon $W$, 
there exists a graph sequence $(G_n(W))_{n \in \mathbb{N}}$ converging to $W$. 

In our paper, a similar path is traced for permutation sequences $(\sigma_n)_{n \in \mathbb{N}}$. 
 The r\^{o}le of the homomorphism density $t(F,G_n)$ of a fixed graph $F$ into $G_n$ is played here by
 the subpermutation density $t(\tau,\sigma_n)$ of a fixed permutation $\tau$ into $\sigma_n$, which we now define. 
By $[n]^m_<$ we mean the set of $m$-tuples in $[n]$ whose elements are in strictly increasing order.

\begin{definition}[Subpermutation density]\label{def.dens.perm1}
For positive integers $k, n \in \N$, let $\tau \in S_k$ and $\pi \in S_n$. 
The \emph{number of occurrences} $\Lambda(\tau,\pi)$  of the permutation $\tau$ in $\pi$ is the number of $k$-tuples
 $(x_1,x_2,\ldots,x_k) \in [n]^k_<$ such that
 $\pi(x_i) < \pi(x_j)$ if and only if $\tau(i)<\tau(j)$. 
The \emph{density} of the permutation $\tau$ as a \emph{subpermutation} of $\pi$ is given by 
\begin{equation}\label{def_subperm_density_formula_cases}
t(\tau,\pi)=
\begin{cases}
\binom{n}{k}^{-1}\Lambda(\tau,\sigma)  & \textrm{if } k\leq n\\
\;0  & \textrm{if  } k> n.
\end{cases}
\end{equation}
\end{definition}

As an illustration, the permutation $\tau=(3,1,4,2)$ occurs in $\pi=(5,6,2,4,7,1,3)$,
since $\pi$ maps the index set $(1,3,5,7)$ onto $(5,2,7,3)$, which appears in the relative order given by~$\tau$. 
This concept may be used to define a convergent permutation sequence in a natural way.
\begin{definition}[Convergence of a permutation sequence]\label{def_conv}
A permutation sequence $(\sigma_n)_{n \in \mathbb{N}}$ is \emph{convergent} if, for 
every fixed permutation $\tau$, the sequence of real numbers $(t(\tau,\sigma_n))_{n \in \mathbb{N}}$ converges.
\end{definition}
The interesting case occurs when the sequence of lengths $(|\sigma_n|)_{n \in \mathbb{N}}$ tends to infinity, since, as we shall see, every convergent permutation sequence $(\sigma_n)_{n \in \mathbb{N}}$ is otherwise eventually constant. We prove that, when $|\sigma_n |\rightarrow \infty$, any convergent permutation sequence has a natural limit object, called a \emph{limit permutation}  (all asymptotics in this paper are with respect to $n \rightarrow \infty$). This limit object consists of a family of \emph{cumulative distribution functions}, or \emph{cdf} for short. We say that a function $F:[0,1] \to [0,1]$ is a cdf if 
\begin{equation}\label{def_cdf}
F \text{ is a non-decreasing and right-continuous function with } F(0) \geq 0 \text{ and } F(1)=1.
\end{equation}
Note that $F$ is a cdf if and only if there is a $[0,1]$-valued random variable $Y$ 
such that, for every $y \in [0,1]$, we have $F(y)=\rbprob{Y\leq y}$.

\begin{definition}[Limit permutation]\label{def.perm.lim}
A \emph{limit permutation} is a Lebesgue measurable function $Z:[0,1]^2 \to[0,1]$ satisfying the following conditions:
\begin{itemize}
\item[(a)] for every $x\in[0,1]$, the function $Z(x,\cdot)$ is a cdf, i.e., \eqref{def_cdf} holds;

\item[(b)] \label{reg_cond_Z_Y_uniform}
 for every $y\in[0,1]$, the function $Z(\cdot,y)$ satisfies
$\int_0^1 Z(x,y)\,dx\ =\ y.$
\end{itemize}
We denote the set of limit permutations by $\mathcal{Z}$.
\end{definition}

In fact, limit permutations are in one-to-one correspondence with
probability measures on the unit square whose marginals are uniformly
distributed in $[0,1]$, as discussed below. The heuristic picture can
again be imagined using TV screens: a permutation $n \in S_n$ might be
represented by the square matrix $A_\sigma=\left( \delta_{\sigma(i),j}
\right)_{i,j=1}^n$, where $\delta_{i,j}$ denotes Dirac's delta
function. Again, if $A(i,j)=1$, we say that the corresponding pixel is
white. A TV screen obtained from a permutation has the property that
the brightness of each row and each column is the same, since each row
and column has a unique white pixel in it.  Now imagine a convergent
sequence of permutations as a sequence of TV screens showing the same
image with higher and higher resolution, and the limit permutation as
the ``perfect TV''.  The brightness distribution of the limit object
can be represented by a probability measure $\mu$ on $[0,1]^2$ and
inherits the property that the total brightness of each ``row'' and
each ``column'' is the same, i.e., the marginal distributions of $\mu$
are uniform on $[0,1]$. Consider the $[0,1]^2$-valued random
variable!$(X,Y)$ with distribution $\mu$.  Denote by
$Z\colon[0,1]^2\to[0,1]$ the \emph{regular conditional distribution
  function} of~$Y$ given~$X$ (for a formal definition of this concept,
see Shiryaev~\cite{shity} and Lemma~\ref{regular} below;
heuristically, we may think of~$Z(x,\cdot)$ as the cdf of $Y$ under
the condition $X=x$). The function~$Z$ will satisfy
Definition~\ref{def.perm.lim}(b) because~$X$ and~$Y$ are uniformly
distributed on~$[0,1]$.  The fact that we may think of limit
permutations as regular conditional distribution functions~$Z$, or as
probability distributions~$\mu$, or as random variables~$(X,Y)$ as
above will be useful in what follows.

As with graphs, limit permutations may be used to define a model of random permutations.
\begin{definition}[$Z$-random permutation]\label{intro_def_of_Z_random_perm_page}
Given  a limit permutation $Z$ we generate the \emph{$Z$-random permutation} $\sigma(n,Z)$ using the following. A sequence of $n$ real numbers $X_1,\dots,X_n$ is generated independently and uniformly on $[0,1]$. Conditional on~$(X_1,\dots,X_n)$, we generate $n$ real numbers $Y_1,\ldots,Y_n$, with each $Y_i$ generated independently according to the cdf $Z(X_i,\cdot)$. Let $(X_1^*,\dots X_k^*)$ and $(Y_1^*,\dots Y_k^*)$ denote the values $X_1, \dots, X_k$ and $Y_1, \dots, Y_k$, respectively, rearranged in increasing order. The $Z$-random permutation $\sigma=\sigma(n,Z)$ is given by $\sigma(i)=j$ if and only if
$X^*_i=X_\ell$ and $Y^*_j=Y_\ell$ for some $\ell \in [n]$.
\end{definition}
In other words, the $Z$-random permutation $\sigma=\sigma(n,Z)$ is given by the relative order of  the vertical coordinates of the points $(X_1,Y_1),\dots,(X_n,Y_n)$ with respect to their horizontal coordinates (we shall see later that, with probability one, $X_i \neq X_j$ and  $Y_i\not= Y_j$ if $i \neq j \in [n]$).

% For example, if $n=3$ and 
%the generation of the $Y_i$ yields $X_1<X_2<X_3$ and $Y_2 < Y_1 <Y_3$, then $\sigma(3,Z)=(2,1,3)$.

This new model of random permutation generalizes the classical random permutation model, 
in which a permutation is selected uniformly at random from all permutations on $[n]$. Indeed, a classical random permutation may be obtained as a $Z$-random permutation for 
the \emph{uniform limit permutation} \label{def_of_Z_u} $Z_u$, where $Z_u(x,y)=y$ for all $(x,y)\in[0,1]^2$. The 
distribution of the corresponding $(X,Y)$  is uniform on $[0,1]^2$, i.e., $X$ and $Y$ are \emph{independent} and uniform
on $[0,1]$.

Again inspired by the graph case, given a limit permutation $Z$, we may define the subpermutation 
density $t(\tau,Z)$ of a permutation $\tau$ on $[k]$ in $Z$ as the probability that the $Z$-random 
permutation $\sigma(k,Z)$ is equal to $\tau$. 
\begin{definition}\label{def:sigma_n_to_Z}
Let  $(\sigma_n)_{n \in \mathbb{N}}$ be a sequence of permutations such that $ |\sigma_n| \to \infty$. 
Let $Z \in \mathcal{Z}$. We say that $(\sigma_n)_{n \in \mathbb{N}}$
converges to $Z$, or briefly write $\sigma_n \to Z$, if 
\begin{equation}\label{intro_conv_t}
 \forall \,  \tau \in \mathcal{S} \, : \, \lim_{n\to\infty}t(\tau,\sigma_n) = t(\tau,Z).
\end{equation}
\end{definition}

Note that the assumption $ |\sigma_n| \to \infty$
 is quite natural: we prove in Claim \ref{claim_eventually_constant} that if  $(\sigma_n)_{n \in \mathbb{N}}$
converges and $(\rbabs{\sigma_n})_{n \in \mathbb{N}}$ has a bounded subsequence then the sequence $(\sigma_n)_{n \in \mathbb{N}}$
is eventually constant.

\begin{theorem}[Main result]\label{teorema_principal}

$ $

\begin{itemize}

\item[(i)] \label{convergent_sequence_has_a_limit}
Given a convergent permutation sequence $(\sigma_n)_{n \in \mathbb{N}}$ for which $ |\sigma_n| \to \infty$, 
there exists a limit permutation $Z$ such that $\sigma_n \to Z$ holds.

\item[(ii)] \label{limit_object_has_a_convergent_seq}
Conversely, every $Z \in \mathcal{Z}$ is a limit of a convergent permutation sequence, i.e., there is a sequence
$(\sigma_n)_{n \in \mathbb{N}}$ such that $\sigma_n \to Z$ holds.
\end{itemize}
\end{theorem}

Theorem \ref{teorema_principal} naturally raises the question of characterizing the limit permutations that are limits to the same given convergent permutation sequence $(\sigma_n)_{n \in \mathbb{N}}$. It is clear that this limit is not unique in a strict sense, because, if $\sigma_n \to Z$ and $A$ is a measurable subset of $[0,1]$ with measure zero, then any limit permutation obtained through the replacement of the cdf $Z(x,\cdot)$ by a cdf $Z^{\ast}(x,\cdot)$ for every $x \in A$ is also a limit of $(\sigma_n)_{n \in \mathbb{N}}$, as the value 
of the probabilities corresponding to $t(\tau,Z)$ and $t(\tau,Z^{\ast})$ will be identical.
This question has also been raised in the case of graph limits in~\cite{borgs06b}, where uniqueness was captured by an equivalence relation induced by a pseudometric $d_{\square}$ between graphons. Two graphons $W$ and $W^{\ast}$ were proved to
 be limits of the same graph sequence if and only if $d_{\square}(W,W^{\ast})=0$. 
The case of permutations is simpler.
\begin{theorem}\label{res_princ1} Let $Z_1, Z_2 \in \mathcal{Z}$.
Let $(\sigma_n)_{n \in \mathbb{N}}$ be a convergent permutation sequence.
Then  $\sigma_n \to Z_1$ and $\sigma_n \to Z_2$  if and only if the set $\{x : Z_1(x,\cdot) \not \equiv Z_2(x,\cdot)\}$
has Lebesgue measure zero.
\end{theorem}

Recall that limit permutations~$Z$ may be viewed as certain
probability distributions~$\mu$ or as certain random variables~$(X,Y)$
(see the discussion just after Definition~\ref{def.perm.lim}).
Theorem~\ref{res_princ1} implies that limit permutations are unique
when viewed as such probability distributions or random variables.

Based on a previous concept by Cooper~\cite{cooper1}, we may introduce
a distance~$d_\square$ between permutations (see \eqref{def.rec})
and, more generally, between limit permutations, which is a
permutation counterpart of the graph pseudometric discussed in the
previous paragraph. In particular, we may characterize our notion of
convergence of permutation sequences in terms of this metric. As
usual, a sequence $(\sigma_n)_{n \in \mathbb{N}}$ is said to be a
\emph{Cauchy sequence} with respect to the metric $d_{\square}$ if,
for every $\eps>0$, there exits $n_0=n_0(\eps)$ such that
$d_{\square}(\sigma_n,\sigma_m)<\eps$ for every $n,m \geq n_0$.

\begin{theorem}\label{teorema_equiv}
A permutation sequence $(\sigma_n)_{n \in \mathbb{N}}$ converges if and only if it is a Cauchy sequence with respect to the metric $d_{\square}$.
\end{theorem}

Also in analogy with the work for graphs by Lov\'{a}sz and Szegedy, the theory
 in this paper can be considered in terms of the discrete metric space $(\mathcal{S},d_{\square})$,
 where $\mathcal{S}=\bigcup_{i=1}^\infty S_n$ is the set of all finite permutations and $d_{\square}$ is
 the metric of the previous paragraph. By a standard diagonalization argument, every permutation 
sequence can be shown to have a convergent subsequence (see Lemma~\ref{conv_subseq}). 
As a consequence, the metric space $(\mathcal{S},d_{\square})$ can be enlarged to a compact metric space $(\mathcal{Z}/_{\sim},d_{\square})$ by adding limit permutations, where we identify limit permutations that are equal almost everywhere. By Theorem~\ref{teorema_principal} the subspace of permutations is dense in $\mathcal{Z}/_{\sim}$; moreover, it is discrete, since a sequence cannot converge to a permutation without being eventually constant (Claim \ref{claim_eventually_constant}).
 Finally, Theorem~\ref{teorema_equiv} tells us that, when restricted to permutations, convergence
 in this metric space coincides with the concept of convergence in Definition~\ref{def_conv}.   

We have found two essentially different paths for establishing the main results in this paper. One of them starkly resembles the work of Lov\'{a}sz and Szegedy~\cite{lovasz06} for graph sequences, which relies on Szemer\'{e}di-type regularity arguments. However, several difficulties of technical nature arise, as the limit objects here are more constrained than in the graph case. For a detailed account of this approach, we refer the reader to~\cite{rudini08c}, while most of its permutation regularity ingredients may be found in~\cite{rudini08b}. 

In this paper, we have opted for an alternative approach, with a
%
% which at first might seem daunting for some combinatorialists due to
% its
%
distinctive probabilistic flavor, which has the advantage of being
both more compact and more direct, since several of the technicalities
of the first approach can be avoided.

The remainder of this paper is structured as follows. In Section
\ref{section:prob_prelim} we collect some preliminary results from
probability theory. For instance, we discuss the relation between
limit permutations and probability measures on the unit square, and we
recall some facts about weak convergence of probability measures in
this context. Moreover, we prove some simple facts about the
convergence of permutation sequences. Section
\ref{section:Z_random_perm_and_subperm_dens} is devoted to a
discussion of $Z$-random permutations and of the probabilistic meaning
of subpermutation densities. Section \ref{section:distance} deals with
the rectangular distance on $\mathcal{S}$ and $\mathcal{Z}$, and we
prove that a large $Z$-random permutation is close to $Z$ in the
$d_\square$-distance with high probability. In Section
\ref{section:limits_of_perm_seq} we define three natural, different
notions of convergence on $\mathcal{Z}$, and prove that they are all
equivalent. This is then used to prove
Theorems~\ref{teorema_principal},~\ref{res_princ1}
and~\ref{teorema_equiv}. For completeness, we provide an appendix
containing the proof of an auxiliary probabilistic lemma that turns
out to be important in our proofs, but which may be proven by 
standard measure-theoretic arguments.

\section{Preliminaries}\label{section:prob_prelim}
The present section introduces the main concepts of probability theory that are important in this paper, and gives the proofs of two simple remarks about the convergence of permutation sequences. It is organized as follows. Subsection \ref{subsection:prob_dist_on_unit_square} deals with relevant properties of probability distributions on the unit square. In Subsection \ref{subsection:weak} we recall useful facts about weak convergence of probability measures on compact subspaces of $\RR^d$, while in Subsection \ref{subsection:reg_cond_prob} we relate limit permutations to probability distributions on the unit square using the notion of regular conditional probability distributions. The remarks about the convergence of permutation sequences are addressed in Subsection~\ref{sub_remarks}. 

\subsection{Probability distributions on the unit square}\label{subsection:prob_dist_on_unit_square} If $A \subseteq \Omega$ is an event in a probability space, we  denote by $\rb_ind[A]$ the indicator of the event, i.e., the random variable which takes value $1$ if $A$ occurs and value $0$ if $A$ does not occur.

We shall work with random variables $(X,Y)$ that take values in the
unit square $[0,1]^2$: $X$ is the horizontal and $Y$ is the vertical
coordinate of the random point $(X,Y)$. The joint distribution of
$(X,Y)$ can be represented in many different ways.

We might define a probability measure $\mu$ on $[0,1]^2$ by defining
$\mu(B)=\rbprob{(X,Y) \in B}$ for every Borel set $B \subseteq
[0,1]^2$, but since the $\sigma$-algebra of Borel sets is generated by
the collection of rectangles of the form $B=[0,x]\times [0,y]$, it is
enough to specify $\mu(B)$ for sets of this form to define $\mu$
uniquely. Thus we define the \emph{joint probability distribution
  function} $F$ of $(X,Y)$ by
\[ F(x,y)=\rbprob{ X \leq x,Y \leq y} =\mu \left( [0,x]\times [0,y]
\right) , \quad x,y \in [0,1]. \] We call the distribution of $X$ and
$Y$ the \emph{first} and \emph{second marginal distributions} of
$(X,Y)$, respectively, which satisfy $ \rbprob{X \leq x} = F(x,1)$ and
$\rbprob{Y \leq y} = F(1,y) $.

If $x_1<x_2$ and $y_1<y_2$ then
\begin{equation}\label{measure_of_rectangle_from_F}
 \rbprob{X \in (x_1,x_2], \, Y \in (y_1,y_2]}=F(x_2,y_2)-F(x_1,y_2)-F(x_2,y_1)+F(x_1,y_1). 
\end{equation}

The $[0,1]^2$-valued random variables that correspond to limit
permutations will always have the property that both $X$ and $Y$ are
uniformly distributed on $[0,1]$, which we denote by $X, \, Y \sim
U[0,1]$.  This happens if and only if we have $F(x,1)=x$, $F(1,y)=y$
for any $x,y \in [0,1]$.

We shall make use of the following simple inequality many times:
\begin{multline}\label{uniform_marginals_bound} 
  X, \, Y \sim U[0,1] \; \implies \;\\
  \forall \; x_1,x_2,y_1,y_2 \in [0,1] \, : \;
  \rbabs{F(x_2,y_2)-F(x_1,y_1)} \leq \rbabs{x_2-x_1}+\rbabs{y_2-y_1}.
\end{multline}
A direct consequence of the fact that~$X$ and~$Y$ are both uniform is
that, for any rectangle
$R=[x_1,x_2] \times [y_1,y_2] \subseteq [0,1]^2$, we have
\begin{equation}\label{boundary_zero_measure}
  \rbprob{ (X,Y) \in \partial R} \leq \rbprob{ X=x_1}+\rbprob{ X=x_2}+
  \rbprob{ Y=y_1}+\rbprob{ Y=y_2} =0,  
\end{equation}
where $\partial R$ denotes the boundary of~$R$.  In particular we have
\begin{equation*}
  F(x,y)=\rbprob{ X<x,Y<y}=\rbprob{ X \leq x,Y \leq y}.
\end{equation*}

If $(X_1,Y_1)$ and $(X_2,Y_2)$ are both $[0,1]^2$-valued random variables then we say that they have the same distribution or briefly write $(X_1,Y_1) \sim (X_2,Y_2)$ if their respective probability measures $\mu_1$ and $\mu_2$ agree on the Borel sets of $[0,1]^2$, or, 
equivalently, if their respective joint distribution functions are the same: $F_1(x,y)=F_2(x,y)$ for all $x,y \in [0,1]$.

\subsection{Weak convergence of probability measures}\label{subsection:weak} Now we recall some well-known facts about weak convergence of probability measures (for details, see \cite{billingsley}).

Let $\Omega$ be a complete, separable metric space and let $\mathcal{B}$ denote the $\sigma$-algebra of Borel sets. If $\mu_1,\mu_2,\ldots$ and $\mu$ are probability measures on the measurable space $(\Omega, \mathcal{B})$, then
 we say that the sequence $(\mu_n)_{n=1}^\infty$ converges 
weakly to $\mu$ (which we denote $\mu_n \Rightarrow \mu$) if for all bounded, continuous functions $f: \Omega \to \RR$ we have
\begin{equation}\label{def_weak_conv_of_prob_meas}
  \lim_{n \to \infty} \int_{\Omega} f(\omega) \, \mathrm{d} \mu_n(\omega)= \int_{\Omega} f(\omega) \, \mathrm{d} \mu(\omega). 
\end{equation}
We will make use of the following consequence of Prokhorov's theorem (see Chapter 1, Section 5 of \cite{billingsley}) characterizing compact subsets of the space of probability 
measures on $\Omega$.
\begin{equation}\label{prokhorov}
\text{ If } \Omega \text{ is compact then every sequence } ( \mu_n)_{n=1}^\infty  \text{ has a weakly convergent subsequence}.
\end{equation}

For $\Omega=\RR^d$, if we denote by $(X^1,\dots,X^d)$ the $\RR^d$-valued random variable with distribution $\mu$, and, similarly, we let
$(X^1_n,\dots,X^d_n)$ be the random variable with distribution $\mu_n$, then we say that $(X^1_n,\dots,X^d_n)$ converges in distribution to
$(X^1,\dots,X^d)$ (or briefly write $(X^1_n,\dots,X^d_n) \toinp (X^1,\dots,X^d)$) if $\mu_n \Rightarrow \mu$.

It easily follows from \eqref{def_weak_conv_of_prob_meas} that
\begin{equation}\label{weak_conv_then_marginals_conv}
 (X^1_n,\dots,X^d_n) \toinp (X^1,\dots,X^d) \qquad \implies \qquad \forall \,  i \in [d] \; : \;
X^i_n \toinp X^i.
\end{equation}
In other words, weak convergence of the joint $d$-dimensional
distributions implies weak convergence of the marginal distributions. As a consequence we obtain the following useful fact about
weak convergence of $[0,1]^2$-valued random variables with uniform marginals:
\begin{equation}\label{weak_conv_uniform_marginals}
\forall \, n \in \NN \, : \, X_n,  \, Y_n \sim U[0,1], \textrm{ and } (X_n,Y_n) \toinp (X,Y) \qquad  \implies \qquad
 X,  \, Y \sim U[0,1].
\end{equation}

The converse implication of \eqref{weak_conv_then_marginals_conv} does not hold in general, but it does hold in the special case when  the marginals are independent. 

Let $\Omega^i$ denote a metric space for all $i \in [d]$, and denote by 
$\Omega=\Omega^1 \times \dots \times \Omega^d$ the product space (we may define the distance between points of $\Omega$ to be the $L_\infty$-distance, that is, the maximum of the distance of each pair of coordinates). If $\mu^i_n$, $\mu^i$ are probability measures on $\Omega^i$ then
\begin{equation}\label{product_measure_weak_conv}
\forall \, i \in [d] \; : \; \mu_n^i \Rightarrow \mu^i \qquad \implies \qquad
\mu_n^1 \times \dots \times \mu^d_n \Rightarrow \mu^1 \times \dots \times \mu^d,
\end{equation}
where $\mu^1 \times \dots \times \mu^d$ denotes the product measure of the measures $\mu^1,\dots,\mu^d$ on the product space $\Omega$. 
For the proof of this fact, see \cite[Theorem 2.8(ii)]{billingsley}.

We say that $B \in \mathcal{B}$ is a continuity set of the measure
$\mu$ if $\mu( \partial B)=0$. The following equivalent
characterization of weak convergence is a consequence of the
\emph{Portmanteau theorem} (see \cite[Theorem 2.1]{billingsley}):
\begin{equation}\label{portemanteau}
  \mu_n \Rightarrow \mu \qquad \iff \qquad \mu_n(B) \to \mu(B) \;
  \text{ for all continuity sets } B \text{ of } \mu. 
\end{equation}

\begin{lemma}\label{convergence_in_distribution_uniform_convergence_of_F}
  Let $(\mu_n)_{n=1}^\infty$ and $\mu$ be probability measures
  on $[0,1]^2$.  Denote by $(X_n,Y_n)$ and $(X,Y)$ the corresponding
  $[0,1]^2$-valued random variables and by $F_n(x,y)$ and $F(x,y)$ the
  corresponding joint probability distribution functions. Assume that
  for all $n \in \N$, we have $X_n, \, Y_n \sim U[0,1]$. Then
  \begin{equation*}
    (X_n,Y_n) \toinq (X,Y) \qquad \iff \qquad \| F_n - F \|_\infty
    =\sup_{x,y \in [0,1]} | F_n(x,y) - F(x,y) | \to 0.  
  \end{equation*}
\end{lemma}
\begin{proof}
  In order to prove that $\|F_n - F \|_\infty \to 0$ implies
  $(X_n,Y_n) \toinp (X,Y)$ we need to
  show~\eqref{def_weak_conv_of_prob_meas}. It is enough to prove that,
  for all $\varepsilon > 0$ and for every continuous function
  $f:[0,1]^2 \to \mathbb{R}$, we have
\begin{equation}\label{weak_conv_on_unit_square}
\limsup_{n \to \infty} \rbabs{\int_0^1 \int_0^1 f(x,y) \, \mathrm{d} \mu_n(x,y)- \int_0^1 \int_0^1 f(x,y) \, \mathrm{d} \mu(x,y)} \leq \varepsilon.
\end{equation}
Since $[0,1]^2$ is compact, $f$ is  uniformly continuous, so
 we can choose  $k \in \N$ such that if we define 
$B_k(i,j)= \left[\frac{i-1}{k}, \frac{i}{k}\right] \times  \left[\frac{j-1}{k}, \frac{j}{k} \right]$ then we have
\begin{equation*} \forall \, i,j \in [k] \; : \;
 \max \left\{ f(x,y)\, :\, (x,y) \in B_k(i,j) \right\}
-  
\min \left\{ f(x,y)\, :\, (x,y) \in B_k(i,j) \right\} 
\leq \varepsilon.
\end{equation*}
We bound the integral on the l.h.s.\ of \eqref{weak_conv_on_unit_square}:
\begin{equation*}
\begin{split}
\int_0^1 \int_0^1& f(x,y) \, (\mathrm{d} \mu_n(x,y)-\mathrm{d}
\mu(x,y)) \\
&\leq \sum_{i,j=1}^k \left( \mu_n(B_k(i,j)) \cdot \max_{B_k(i,j)} f  -
  \mu(B_k(i,j)) \cdot \min_{B_k(i,j)} f\right) \\
&\leq 
\sum_{i,j=1}^k \varepsilon \cdot \mu_n(B_k(i,j))
+ \sum_{i,j=1}^k \rbabs{\rbabs{f}}_\infty \cdot
\rbabs{\mu_n(B_k(i,j))-\mu(B_k(i,j))}\\  
&\stackrel{\eqref{measure_of_rectangle_from_F}}{\leq} 
\varepsilon + k^2 \|f\|_\infty \cdot 4 \|F_n-F\|_\infty.
\end{split}
\end{equation*}
With the same arguments, we may prove that 
$${\int_0^1 \int_0^1 f(x,y) \, (\mathrm{d} \mu(x,y)-\mathrm{d} \mu_n(x,y))} \leq \varepsilon + k^2 \|f\|_\infty \cdot 4 \|F_n-F\|_\infty,$$
so that
$$\rbabs{\int_0^1 \int_0^1 f(x,y) \, (\mathrm{d} \mu_n(x,y)-\mathrm{d} \mu(x,y))} \leq \varepsilon + k^2 \|f\|_\infty \cdot 4 \|F_n-F\|_\infty.$$
Now  \eqref{weak_conv_on_unit_square} follows from the above inequality and $\| F_n - F \|_\infty \to 0$.

\medskip

In order to prove that $ (X_n,Y_n) \toinp (X,Y)$ implies $\| F_n - F \|_\infty \to 0 $, we will have to make use of our assumption that $ X_n, \, Y_n \sim U[0,1]$ for every $n$ (this implication does not hold for general sequences of weakly convergent probability measures on $[0,1]^2$). 

We are going to show that
\begin{equation}\label{5_over_k}
\forall \; k \in \NN  \;\; \exists \; n_0 \in \N \;\; \forall \; n \geq n_0 \; : \; 
\| F_n - F \|_\infty=\sup_{x,y \in [0,1]} \rbabs{F_n(x,y)-F(x,y)} \leq \frac{5}{k}.
\end{equation}

Since $ X_n, \, Y_n \sim U[0,1]$ for every $n$ and $(X_n,Y_n) \toinp (X,Y)$, we know from \eqref{weak_conv_uniform_marginals} that $X,Y \sim U[0,1]$. By \eqref{boundary_zero_measure}, it follows that, for every rectangle $R \in [0,1]^2$, $\rbprob{ (X,Y) \in \partial R} = 0$, so that rectangles are continuity sets with respect to the associated measure $\mu$. We may then apply
\eqref{portemanteau} to deduce that $F_n(x,y) \to F(x,y)$ for all $x,y \in [0,1]$. Thus there exists $n_0 \in \N$ such that for all $n \geq n_0$ and $i,j \in \{0,1,\dots,k\}$ we have 
\begin{equation}\label{weak_convergence_on_finite_lattice}
\rbabs{F_n \left( \frac{i}{k},\frac{j}{k} \right) - F \left( \frac{i}{k},\frac{j}{k} \right)} \leq \frac{1}{k}
\end{equation}
In order to deduce \eqref{5_over_k} from this, pick $x,y \in [0,1]$ and let $i:=\lfloor k \cdot x \rfloor$ and $j:=\lfloor k \cdot y \rfloor$. Then
\begin{equation}\label{epsilon_over_three_trick}
\begin{split}
 &\rbabs{F_n \left(x,y\right)-F\left(x,y\right)} \\
&~~~~\leq \rbabs{F_n\left(x,y\right)-F_n\left(\frac{i}{k},\frac{j}{k}\right)} + 
\rbabs{F_n\left(\frac{i}{k},\frac{j}{k}\right)-F\left(\frac{i}{k},\frac{j}{k}\right) } + \rbabs{F\left(x,y\right)-
F\left(\frac{i}{k},\frac{j}{k}\right)}  \\
&~~~~\stackrel{\eqref{uniform_marginals_bound}}{\leq} \frac{2}{k} + \rbabs{F_n\left(\frac{i}{k},\frac{j}{k}\right)-F\left(\frac{i}{k},\frac{j}{k}\right) }+ \frac{2}{k} 
\stackrel{\eqref{weak_convergence_on_finite_lattice}}{\leq} \frac{5}{k},
\end{split}
\end{equation}
which concludes the proof.
\end{proof} 

\subsection{Regular conditional probabilities}\label{subsection:reg_cond_prob} The aim of this subsection is to show that limit permutations are (essentially) in one-to-one correspondence with probability measures on the unit square with $U[0,1]$ marginals:  $Z$ is the \emph{regular conditional distribution function} of $Y$ with respect to $X$. Heuristically we may think about $Z$ as the function satisfying  $Z(x,y)=\rbcondprob{Y \leq y}{X=x}$, although the condition $X=x$ has zero probability.

\medskip

Denote by $\mathcal{B}[0,1]$ the $\sigma$-algebra of Borel sets of the
unit interval $[0,1]$ and recall the definition of limit permutation,
given in Definition~\ref{def.perm.lim}.  

\begin{lemma}\label{regular}
  Let $(X,Y)$ be a $[0,1]^2$-valued random variable satisfying $X, \,
  Y \sim U[0,1]$.
\begin{itemize}
\item[(\textit{a})] Then there exists a limit permutation $Z$ such that
\begin{equation}\label{reg_cond_borel_formula}
\forall \, B \in \mathcal{B}[0,1], \; \; \forall y \in [0,1] \; : \; 
  \rbprob{X \in B,\, Y \leq y}=\int_0^1  Z(x,y) \rb_ind[ x \in B]   \, {d} x.
\end{equation}
\item[(\textit{b})] Moreover, if $\ykwhat{Z}$ is another limit permutation satisfying \eqref{reg_cond_borel_formula} then
\begin{equation}\label{uniqueness_of_regular_condexp}
 \int_0^1 \rb_ind[\, \exists \, y \in [0,1] \, : \, Z(x,y)\neq \ykwhat{Z}(x,y)\, ] \, {d}x=0. 
\end{equation}
\end{itemize}
\end{lemma}

The proof of Lemma~\ref{regular} is a standard measure-theoretic
argument that is almost identical to that of Theorem~4 of Chapter~II,
Section~7 of \cite{shity} about the existence of \emph{regular
  conditional distribution functions}. Nevertheless we give a full
proof of Lemma~\ref{regular} in the appendix in order to keep the
paper self-contained.

\medskip

Lemma \ref{regular} tells us how to obtain a limit permutation $Z$ from a $[0,1]^2$-valued random variable $(X,Y)$ satisfying $X, \, Y \sim U[0,1]$.
 Conversely, with a limit permutation $Z$, we may associate such a random variable $(X,Y)$ as follows.
\begin{definition}\label{def:X_Y_associated_to_Z}
 Let $Z \in \mathcal{Z}$. We generate the $[0,1]^2$-valued random variable $(X,Y)$ \emph{associated with} $Z$ in the following way.We first
pick $X \sim U[0,1]$, and, given $X$, the random variable $Y$ is generated according to the cdf $Z(X,\cdot)$. 
\end{definition}

It is easy to check that with this definition $Z$ is indeed the regular conditional distribution
function of $Y$ with respect to $X$ in the sense of Lemma \ref{regular}.

The joint distribution function of $(X,Y)$ associated with $Z$ can be expressed as
\begin{equation}\label{distribution_fn}
F(x,y)= \rbprob{ X \leq x,\, Y \leq y}  \stackrel{\eqref{reg_cond_borel_formula}}{=}\int_0^x Z(x,y)\, {d}x.
\end{equation}
 By Definition \ref{def.perm.lim} (b) we get $F(1,y)=y$ which is
equivalent to $Y \sim U[0,1]$.

\subsection{Two simple remarks about permutation convergence}\label{sub_remarks} In this subsection, we prove two simple facts about the convergence of permutation sequences that were mentioned in the introduction. We show that every convergent sequence $(\sigma_n)_{n \in \N}$ such that $|\sigma_n| \not \to \infty$ must be eventually constant, and we establish that every permutation sequence contains a convergent subsequence.

For the first, we remind the reader that Theorem~\ref{teorema_principal} is only stated for permutation sequences $(\sigma_n)_{n \in \N}$ whose lengths tend to infinity, as we claimed that every other convergent sequence is eventually constant. In light of this, we prove this claim prior to addressing the main results. As before, if $\sigma \in S_n$ we write $\rbabs{\sigma}=n$. Recall the notion of convergence of permutation sequences from Definition \ref{def_conv}. 
\begin{claim}\label{claim_eventually_constant}
 Let $(\sigma_n)_{n \in \mathbb{N}}$ be a convergent permutation sequence such that $|\sigma_n| \not \to \infty$. Then
 the  sequence $(\sigma_n)_{n \in \mathbb{N}}$ is eventually constant, that is,
there is a permutation $\sigma$ and an $n_0 \in \N$ such that $n \geq n_0$ implies $\sigma_n = \sigma$.
\end{claim}
\begin{proof}
  It follows from \eqref{def_subperm_density_formula_cases} that 
  $\sum_{\tau \in S_k} t(\tau, \pi)=\rb_ind [ k \leq
  \rbabs{\pi}]$ for any $k \in \N$ and any permutation~$\pi$. By
  Definition \ref{def_conv} we get that for any fixed $k \in \N$ the
  limit $\lim_{n \to \infty} \sum_{\tau \in S_k} t(\tau,
  \sigma_n)=\lim_{n \to \infty} \rb_ind [ k \leq \rbabs{\sigma_n} ] $
  exists and must be equal to $0$ or $1$.

  From this and our assumption that $\liminf_{n \to \infty}
  \rbabs{\sigma_n}<\infty$ one deduces that $\liminf_{n \to
    \infty} \rbabs{\sigma_n}=\limsup_{n \to \infty} \rbabs{\sigma_n}$;
  thus there is some $m \in \N$ such that $\rbabs{\sigma_n}=m$ if $n$
  is large enough.

  Now if $\tau, \pi \in S_m$ then $t(\tau,\pi)=\rb_ind[\tau=\pi]$, and
  hence $\lim_{n \to \infty} t(\tau, \sigma_n)$ must be equal to~$0$
  or~$1$ for all $\tau \in S_m$.  From this and
  $\rbabs{S_m}=m!<\infty$ it is straightforward to deduce that the
  sequence $(\sigma_n)_{n \in \mathbb{N}}$ is eventually constant.
\end{proof}

To conclude this section, we show that permutation sequences always contain convergent subsequences.
\begin{lemma}\label{conv_subseq}
Every permutation sequence has a convergent subsequence.
\end{lemma}

\begin{proof}
Let $(\sigma_n)_{n \in \mathbb{N}}$ be a permutation sequence.
 We shall find a convergent subsequence of $(\sigma_n)_{n \in \mathbb{N}}$ by a standard diagonalization argument.
 Since $S_n$ is finite for every $n \geq 1$, the set of $S=\bigcup_{n=1}^\infty S_n$ of all finite permutations is countable,
 say $S=(\tau_m)_{m \in \mathbb{N}}$.

If $(\sigma_n)_{n \in \mathbb{N}}$ does not converge, starting with $\tau_1$,
 we let $(\sigma_n^1)_{n \in \mathbb{N}}$ be a subsequence of $(\sigma_n)_{n \in \mathbb{N}}$ for which the bounded real
 sequence $(t(\tau_1,\sigma^1_n))_{n \in \mathbb{N}}$ converges. Inductively, for $m \geq 2$, we let $(\sigma_n^m)_{n \in \mathbb{N}}$ be 
a subsequence of $(\sigma_n^{m-1})_{n \in \mathbb{N}}$ such that $(t(\tau_m,\sigma^{m-1}_n))_{n \in \mathbb{N}}$ converges. 
It is now easy to see that the diagonal sequence $(\sigma_n^n)_{n \in \mathbb{N}}$ is such that, for every positive integer $m$, 
the sequence $(t(\tau_m,\sigma_n^n))_{n \in \mathbb{N}}$ converges. In other words, the sequence $(\sigma_n^n)_{n \in \mathbb{N}}$ is 
a convergent subsequence of $(\sigma_n)$.
\end{proof}

\section{$Z$-random permutations and subpermutation densities}\label{section:Z_random_perm_and_subperm_dens}

In this section we define the concept of a random subpermutation $\sigma(k,\pi)$ of length $k$ of a permutation $\pi$ and relate
the subpermutation densities $t(\tau,\pi)$ to the distribution of  $\sigma(k,\pi)$. 
Analogously, given a limit permutation $Z$ we define the $Z$-random permutation $\sigma(k,Z)$ of length $k$  and
the subpermutation densities $t(\tau,Z)$. In order to treat permutations and limit permutations in a unified way we assign a limit 
permutation $Z_\sigma$ to every permutation $\sigma$ in such a way that their subpermutation densities are close to each other.

 \medskip

First we recall Definition \ref{def.dens.perm1}, the definition of the density $t(\tau,\pi)$ of $\tau \in S_k$ in
$\pi \in S_n$. Now we give a probabilistic interpretation of this quantity.

\begin{definition}\label{def:random_subpermutation}
 Let $\pi \in S_n$. For $k \leq n$,  the \emph{random subpermutation} $\sigma(k,\pi)$ of $\pi$ of \emph{length} $k$ is the  random element of $S_k$  generated in the following way. Choose an element
 $(\mathcal{X}^*_1,\ldots,\mathcal{X}^*_k) \in [n]^k_<$ uniformly from the $\binom{n}{k}=\rbabs{ [n]^k_<}$ possibilities and 
let $\sigma=\sigma(k,\pi)$
be the permutation of $[k]$ satisfying
  \[ \forall \, i,j \in [k] \; : \;  \pi(\mathcal{X}^*_i) < \pi(\mathcal{X}^*_j)\; \iff \;  \sigma(i)<\sigma(j).\] 
\end{definition}
In plain words: $\sigma(k,\pi)$ encodes the relative order of $(\pi(\mathcal{X}^*_1),\dots,\pi(\mathcal{X}^*_k))$. It is obvious from 
Definition \ref{def.dens.perm1} that we have
\begin{equation}\label{probabilistic_interpretation_of_t_tau_sigma}
\forall \, \tau \in S_k \; : \;  \rbprob{ \sigma(k,\pi) =\tau} = t(\tau,\pi).
\end{equation}

\begin{definition}[$Z$-random permutation (Version 2)]
\label{def:Z_random_perm}
Given a limit permutation $Z$ and a positive integer $n$, a \emph{$Z$-random permutation} $\sigma(n,Z)$ is a permutation of $[n]$ generated as follows. Recall Definition \ref{def:X_Y_associated_to_Z}, and let $(X_1,Y_1),\ldots,(X_n,Y_n)$ be independent and identically distributed $[0,1]^2$-valued random variables with distribution associated with $Z$.  
These pairs define the permutations $R,S \in S_n$,
where 
\begin{equation}\label{def_R_S}
R(i)= |\{j~:~X_j \leq X_i\}|, \qquad   S(i)= |\{j~:~Y_j \leq Y_i\}|.
\end{equation}
The random permutation $\sigma=\sigma(n,Z)$ is given by $\sigma(n,Z)=S \circ R^{-1}$, that is, $\sigma(i)=S(R^{-1}(i))$ for every $i\in[n]$.
\end{definition}
Note that  $Z$-random permutations are well-defined with probability one, because the probability that either $(X_1,\ldots,X_n)$ or $(Y_1,\ldots,Y_n)$ has repeated elements (i.e., $S$ or $R$ is not a well-defined permutation) is zero. This follows directly from the fact that both $(X_1,\dots,X_n)$ and $(Y_1,\dots,Y_n)$ are independent and uniformly distributed on
$[0,1]$. It is easy to see that $\sigma(n,Z)$ in Definition \ref{def:Z_random_perm}  is equivalent to the one given in Definition~\ref{intro_def_of_Z_random_perm_page}, 
since the $X^*_i$ and $Y^*_j$ defined in the latter may be related to $X$ and $Y$ by $X_{R^{-1}(i)}=X^*_i$ and $Y_{j}=Y^*_{S(j)}$.
%As an example, suppose that $n=4$, $X=(0.7,0.3,0.9,0.2)$ and $Y=(0.8,0.1,0.5,0.3)$.
%Then the permutations $R$ and $S$ are given by $R=(3,2,4,1)$ and $S=(4,1,3,2)$, so that %$\sigma=S\cdot R^{-1}=(2,1,4,3)$.
Moreover, observe that one could alternatively define a random permutation with the same distribution as $\sigma(n,Z)$ by first generating a  sequence $(X^*_1,\dots,X^*_n)$  uniformly distributed on $[0,1]^n_<$ and then drawing each $\ykwhat{Y}_i$ in $[0,1]$ independently according to the probability distribution induced by $Z(X^*_i,\cdot)$. The random permutation $\sigma(n,Z)$ is given by the order of the elements in $(\ykwhat{Y}_1,\ldots,\ykwhat{Y}_n)$. 
This definition of $\sigma(n,Z)$ resembles Definition \ref{def:random_subpermutation}.

\begin{definition}\label{def:subperm_density_in_Z}
Given a limit permutation $Z$ and $\tau \in S_k$ the \emph{density} of $\tau$ \emph{in} $Z$ is given by
\begin{equation}\label{formula_subperm_density_in_Z}
t(\tau,Z)=\rbprob{\sigma(k,Z)=\tau}. 
\end{equation}
\end{definition}

It will be convenient for us to associate a limit permutation $Z_\sigma$ with every permutation $\sigma$ in such a way that $t(\tau,Z_\sigma)$ of \eqref{formula_subperm_density_in_Z}
is close to $t(\tau,\sigma)$ of \eqref{probabilistic_interpretation_of_t_tau_sigma}.

\begin{definition}\label{def:Z_from_perm}
Given a permutation $\sigma \in S_n$ define  $Z_\sigma \in \mathcal{Z}$ by defining the associated
$[0,1]^2$-valued random $(X_\sigma, Y_\sigma)$ to have joint density function 
\begin{equation}\label{density_function_of_perm}
f_\sigma(x,y)= n \cdot \rb_ind [ \; \sigma( \lceil n \cdot x \rceil) =   \lceil n \cdot y \rceil \; ].
\end{equation}
 \end{definition}
Note that the support of the density function $f_\sigma(x,y)$ ``looks like'' the matrix $A_\sigma=\left( \delta_{\sigma(i),j} \right)_{i,j=1}^n$
 of the permutation $\sigma$, the row and column sums
of which are all equal to 1. From this it easily follows that we  have $X_\sigma, Y_\sigma \sim U[0,1]$; thus, by Lemma \ref{regular} we can define $Z_\sigma$ almost surely uniquely. Indeed,
\begin{equation}\label{ze_szigma}
Z_{\sigma}(x,y)=\int_0^y f_{\sigma}(x,\tilde{y})\, {d}\tilde{y}.
\end{equation}
 
\begin{lemma}\label{permutation_and_associated_limit_permutation_subperm_densities_are_close}
  Let $\tau \in S_k$, $\sigma \in S_n$, $k \leq n$. Then
\begin{equation}\label{smoothing_small_change}
 \rbabs{ t(\tau, \sigma) -t(\tau,Z_{\sigma})} \leq \frac{1}{n}  \binom{k}{2}.
  \end{equation}
\end{lemma}

\begin{proof}
If $A,B$ are events in a probability space then it is easy to check that
\begin{equation}\label{condition_can_be_dropped_if_large}
 \rbabs{\rbprob{A} -\rbcondprob{A}{B}} \leq 1-\rbprob{B} =\rbprob{B^c}.
 \end{equation}
We are going to apply this inequality to prove \eqref{smoothing_small_change}. To this end, let $(X_i,Y_i)$, $i=1,\dots,k$, be i.i.d.\,with the same joint distribution as $(X_\sigma,Y_\sigma)$
of Definition \ref{def:Z_from_perm}. Then $X_i$, $i=1,\dots,k$, are i.i.d. and uniform on $[0,1]$.

Let $A$ be the event that the relative order of the vertical 
coordinates of $(X_i,Y_i)$, $i=1,\dots,k$ is $\tau$ with respect to their ordered horizontal coordinates, i.e., with the notation of \eqref{def_R_S}, let $A=\{ \tau=S \circ R^{-1} \}$.

Define $\mathcal{X}_i:=\lceil n \cdot X_i \rceil$ for each $i \in [k]$.
We define the event
\begin{equation}\label{def_event_B}
B= \{ \forall\; 1\leq i < j \leq k\,:\, \mathcal{X}_i \neq \mathcal{X}_j \}
\end{equation}
 and let
$(\mathcal{X}^*_1,\mathcal{X}^*_2,\ldots,\mathcal{X}^*_k)$ denote the $k$-tuple that we get by arranging 
 $(\mathcal{X}_1,\mathcal{X}_2,\ldots,\mathcal{X}_k)$ in increasing order. Note that $\lceil n \cdot Y_i \rceil=\sigma(\mathcal{X}_i)$ by \eqref{density_function_of_perm}, and that
under the condition that the event $B$ occurs 
the random $k$-tuple
$(\mathcal{X}^*_1,\mathcal{X}^*_2,\ldots,\mathcal{X}^*_k)$ is uniformly distributed on $[n]^k_<$.
 Thus we have 
 \[ \rbprob{A}\stackrel{\eqref{formula_subperm_density_in_Z}}{=}t(\tau,Z_{\sigma}),
 \quad \rbcondprob{A}{B} \stackrel{\eqref{probabilistic_interpretation_of_t_tau_sigma}}{=} t(\tau,\sigma) \quad \textrm{and} 
 \quad \rbprob{B^c}\leq \frac{1}{n}  \binom{k}{2},\]
and our result follows from \eqref{condition_can_be_dropped_if_large}.
\end{proof}

\section{Rectangular distance}\label{section:distance}

In this section we introduce the notion of rectangular distance $d_{\square}(\sigma_1,\sigma_2)$ between two permutations and also give the
analogous definition $d_{\square}(Z_1,Z_2)$ of rectangular distance between two limit permutations. Using the notion of the limit permutation $Z_\sigma$ assigned to a permutation $\sigma$, we make a connection between these two definitions of $d_{\square}$.

In Subsection \ref{subsection:rectang_subperm} we prove that the rectangular distance between $Z$ and the random subpermutation 
$\sigma(k, Z)$ is small with high probability if $1 \ll k$. In other words, we show that $Z$ can be recovered from the sample $\sigma(k, Z)$ with small error. From this
we derive an analogous statement which quantifies how small $d_{\square}(\pi,\sigma(k,\pi))$ is. This will prove to be useful in the characterization of testability in~\cite{rudini08d}.

\subsection{Preliminary facts about the rectangular distance} Let $I[n]$ be the set of all \emph{intervals} in $[n]$, that is, the set of all subsets of the form $\{x \in [n]~:~ a \leq x < b\}$,
 where $a,b \in [n+1]$ are called the \emph{endpoints} of the interval. Given a permutation $\sigma$ on $[n]$, Cooper~\cite{cooper2}
 defines the \emph{discrepancy of $\sigma$} as
\begin{equation}\label{cooper_discrepancy}
 D(\sigma)=\max_{S,T\in I[n]}  \left||\sigma(S)\cap T|-\frac{|S||T|}{n}\right|.
\end{equation}

This is used to measure the ``randomness'' of a permutation. Indeed, sequences with low discrepancy, i.e.,
 for which $D(\sigma)=o(n)$, are said to be \emph{quasi-random}.
 We use a normalized version of the same concept to introduce a distance between permutations.
\begin{definition}
Given permutations $\sigma_1,\sigma_2 \in S_n$, the \emph{rectangular distance} between $\sigma_1$ and $\sigma_2$ is given by
\begin{equation}\label{def.rec}
  d_{\square}(\sigma_1,\sigma_2)\ =\ \frac{1}{n}\max_{S,T\in I[n]}\ \Big||\sigma_1(S)\cap T|-|\sigma_2(S)\cap T|\Big|.
\end{equation}
\end{definition}
An analogous metric  $d_{\square}$ may be defined to measure the distance between   limit permutations. Given  $Z_1,Z_2 \in \mathcal{Z}$, and the associated $[0,1]^2$-valued random variables $(X_1,Y_1)$ and $(X_2,Y_2)$ 
(see Definition \ref{def:X_Y_associated_to_Z}),
the \emph{rectangular distance} between $Z_1$ and $Z_2$ is defined by
\begin{multline}\label{rectangular_distance_def_eq}
  d_{\square}(Z_1,Z_2)\ =\ \sup_{\substack{x_1<x_2\in [0,1]\\y_1<y_2\in[0,1]}}\
  \Big|\int_{x_1}^{x_2} \left( Z_1(x,y_2)-Z_1(x,y_1)\right) {d}x
  - \int_{x_1}^{x_2} \left(Z_2(x,y_2)-Z_2(x,y_1)\right) {d}x \Big|  \\
\stackrel{\eqref{reg_cond_borel_formula},\eqref{boundary_zero_measure}}{=} \sup_{\substack{x_1<x_2\in [0,1]\\y_1<y_2\in[0,1]}}\
  \Big|\rbprob{X_1 \in [x_1,x_2], \, Y_1 \in [y_1,y_2]}
  -  \rbprob{X_2 \in [x_1,x_2], \,  Y_2 \in [y_1,y_2]} \Big|.
\end{multline}

Denote by $F_1$ and $F_2$ the joint probability distribution function (see \eqref{distribution_fn}) of $(X_1,Y_1)$ and $(X_2,Y_2)$, respectively. 
We define
\begin{equation}\label{def_infty_norm}
d_\infty(Z_1,Z_2)=\|F_1-F_2\|_\infty=
\sup_{x,y \in [0,1]} 
  \Big|\rbprob{X_1 \leq x, \, Y_1 \leq y}
  -  \rbprob{X_2 \leq x, \,  Y_2 \leq y} \Big|.
\end{equation}
Using identity \eqref{measure_of_rectangle_from_F}, it is easy to deduce that
\begin{equation}\label{d_infty_d_square}
  d_\infty(Z_1,Z_2) \leq  d_{\square}(Z_1,Z_2) \leq 4\cdot d_\infty(Z_1,Z_2).
\end{equation}
 In the sequel we will use the simpler $d_\infty$ in our proofs rather than $d_{\square}$. 
Note that we have
\begin{multline}\label{equivalence_of_identities}
 d_{\square}(Z_1,Z_2)=0 \; \iff \; d_\infty(Z_1,Z_2)=0 \; \iff \; F_1 \equiv F_2 \; \iff \;
 \\ (X_1,Y_1) \sim  (X_2,Y_2)\; \stackrel{\eqref{uniqueness_of_regular_condexp}}{\iff} \; 
\int_0^1 
\rb_ind [ Z_1(x,\cdot) \equiv Z_2(x,\cdot) ]
 \, {d}x =1.
\end{multline}

Recalling Definition \ref{def:Z_from_perm} it is not hard to see that, for permutations $\sigma_1, \sigma_2 \in S_n$,
 we have $d_{\square}(\sigma_1,\sigma_2)=d_{\square}(Z_{\sigma_1},Z_{\sigma_2})$,
 which allows us to extend the definition of rectangular distance to permutations
 on different sets of integers. Indeed, we may define $d_{\square}(\sigma,\pi):=d_{\square}(Z_{\sigma},Z_{\pi})$ for every pair of permutations $\sigma, \pi \in \mathcal{S}$.

Similarly, we define the rectangular distance of a permutation $\sigma$ and a limit permutation $Z$ by
\begin{equation}\label{rect_dist_between_perm_and_limperm}
 d_{\square}(\sigma,Z):=d_{\square}(Z_{\sigma},Z).
\end{equation}
 With this definition we may express the discrepancy in \eqref{cooper_discrepancy} as
 $D(\sigma)=n \cdot d_{\square}(\sigma, Z_u)$, where $Z_u$ is the uniform limit permutation defined on page \pageref{def_of_Z_u}.

\subsection{Rectangular distance and subpermutations}\label{subsection:rectang_subperm}
The main objective of this subsection is to show that the rectangular distance between a limit permutation and a large constant-size random subpermutation $\sigma(k,Z)$ of it is small.
\begin{lemma}\label{limit_perm_and_large_random_perm_are_close_in_rectangular}
If $k \in \N$ is a sufficiently large integer and $Z\in \mathcal{Z}$, then 
\begin{equation}
 \rbprob{d_{\square} \left( Z,\sigma(k,Z) \right) \leq 16 k^{-1/4}} \geq 1-\frac{1}{2} e^{-\sqrt{k}}.
\end{equation}
 \end{lemma}

\begin{proof}
As in \eqref{distribution_fn}, denote by $F$ the joint probability distribution function of the random variable $(X,Y)$ associated with the limit permutation $Z$ (see Definition \ref{def:X_Y_associated_to_Z}). Let $\sigma_k=\sigma(k,Z)$ and, recalling \eqref{density_function_of_perm}, we define the random functions $f_k$ and $F_k$ by
\begin{equation}\label{random_perm_dens_dist_k}
 f_k(x,y)= n \cdot \rb_ind [ \; \sigma_k( \lceil n \cdot x \rceil) =   \lceil n \cdot y \rceil \; ], \quad
F_k(x,y)=\int_0^y \int_0^x f_k(u,v) \, {d} v \; {d} u.
\end{equation}
By \eqref{rect_dist_between_perm_and_limperm}, \eqref{def_infty_norm} and \eqref{d_infty_d_square}, we only need to prove that if $k$ is large enough then 
\begin{equation}\label{dist_of_Z_and_sample_distributionfunction}
 \rbprob{\sup_{x,y \in [0,1]}\rbabs{F(x,y)-F_k(x,y)}  > 4 k^{-1/4}} \leq \frac12 e^{-\sqrt{k}}.
\end{equation}
As in \eqref{epsilon_over_three_trick}, for $i=\lfloor k \cdot x \rfloor$ and $j=\lfloor k \cdot y \rfloor$, we have 
$ \rbabs{F(x,y)-F_k(x,y)} \leq \frac{4}{k} + \rbabs{F\left( \frac{i}{k},\frac{j}{k}\right)-F_k\left( \frac{i}{k},\frac{j}{k}\right)} $.
We claim that, in order to show \eqref{dist_of_Z_and_sample_distributionfunction}, we only need to show that for large $k$ we have
\begin{equation}\label{dist_of_Z_and_sample_lattice}
 \rbprob{\max_{i,j \in [k]} 
\rbabs{F\left( \frac{i}{k},\frac{j}{k}\right)-F_k\left( \frac{i}{k},\frac{j}{k}\right)}
 > 3 k^{-1/4}} < 6e^{-2\sqrt{k}}.
\end{equation}
This is because $4/k<k^{-1/4}$ for large $k$ and
\begin{multline*}
\prob{\max_{i,j \in [k]} 
\abs{F\left( \frac{i}{k},\frac{j}{k}\right)-F_k\left( \frac{i}{k},\frac{j}{k}\right)}
 > 3 k^{-1/4} } = \\
\prob{ \bigcup_{i,j \in [k]} \left\{ \abs{F\left( \frac{i}{k},\frac{j}{k}\right)-F_k\left( \frac{i}{k},\frac{j}{k}\right)}
 > 3 k^{-1/4} \right\} } {\leq} \\
\sum_{i,j \in [k]} \prob{  \abs{F\left( \frac{i}{k},\frac{j}{k}\right)-F_k\left( \frac{i}{k},\frac{j}{k}\right)}
 > 3 k^{-1/4}  } \stackrel{\eqref{dist_of_Z_and_sample_lattice}}{\leq} k^2 \cdot 6 e^{-2\sqrt{k}}=
(12k^2e^{-\sqrt{k}})\cdot \left(\frac12 e^{-\sqrt{k}}\right).
\end{multline*}
Now $12 k^2  e^{-\sqrt{k}}\leq 1$ if $k$ is large enough, and we get \eqref{dist_of_Z_and_sample_distributionfunction}.

To establish~\eqref{dist_of_Z_and_sample_lattice}, let $(X_1,Y_1)$, \dots, $(X_k,Y_k)$ be i.i.d.\,$[0,1]^2$-valued random variables with distribution function $F(x,y)$. Denote by $(X_1^*,\dots X_k^*)$ and  $(Y_1^*,\dots Y_k^*)$ the values of $\{X_1, \dots, X_k\}$ and $\{Y_1, \dots, Y_k\}$, respectively, rearranged in increasing order.

Then we have
\begin{equation*}
 F_k\left( \frac{i}{k},\frac{j}{k}\right) \stackrel{\eqref{random_perm_dens_dist_k}}{=}
\frac{1}{k} \sum_{\ell=1}^i \rb_ind[ \sigma_k(\ell) \leq j] 
= \frac{1}{k} \sum_{\ell=1}^k \rb_ind [ X_\ell \leq X^*_i, Y_\ell \leq Y^*_j].
\end{equation*}
In order to prove \eqref{dist_of_Z_and_sample_lattice} we first show that 
\begin{equation}\label{dist_of_Z_and_sample_indicator_ordered}
 \forall\; i,j \in [k] \, : \; 
\rbprob{ 
\frac{1}{k} \sum_{l=1}^k \rb_ind [ X_l \leq X^*_i, Y_l \leq Y^*_j]
 > F\left( \frac{i}{k},\frac{j}{k}\right) +  3 k^{-1/4}} < 3 e^{-2\sqrt{k}}.
\end{equation}
Set $\varepsilon=k^{-1/4}$. By \eqref{uniform_marginals_bound} we get
\begin{equation}\label{uniform_marginals_epsilon_trick}
F\left( \frac{i}{k}+\varepsilon, \frac{j}{k}+\varepsilon \right) +\varepsilon \leq 
 F\left( \frac{i}{k}, \frac{j}{k} \right) + 3 \varepsilon.
\end{equation}
Let us define the event 
$$A:=\left\{ \frac{1}{k} \sum_{l=1}^k \rb_ind [ X_l \leq X^*_i, Y_l \leq Y^*_j]
 > F\left( \frac{i}{k}+\varepsilon, \frac{j}{k}+\varepsilon \right) +\varepsilon \right\}.$$
Then by 
\eqref{uniform_marginals_epsilon_trick}
 the probability on the left-hand side of
\eqref{dist_of_Z_and_sample_indicator_ordered} is less then or equal to $\rbprob{A}$. On the other hand,
\begin{multline}\label{prob_A_in_three_pieces}
\rbprob{A} =\\ \rbprob{A \cap \{ X^*_i < \frac{i}{k} +\varepsilon \} \cap \{ Y^*_j < \frac{j}{k} +\varepsilon \} }
+ \rbprob{A \cap  \left(\{ X^*_i \geq \frac{i}{k} +\varepsilon\} \cup \{ Y^*_j \geq \frac{j}{k} +\varepsilon \} \right) }\\
 \leq \rbprob{A \cap \{ X^*_i < \frac{i}{k} +\varepsilon\} \cap \{ Y^*_j < \frac{j}{k} +\varepsilon \} } +\rbprob{ X^*_i > \frac{i}{k} +\varepsilon} +\rbprob{ Y^*_j > \frac{j}{k} +\varepsilon}.
\end{multline}
In order to bound these probabilities we are going to use the following large deviation estimate on binomial random variables, which is a consequence of Hoeffding's inequality (see McDiarmid~\cite{mcdiarmid98:_concen}). 
\begin{equation}\label{hoeffding}
S \sim \text{BIN}(k,p) \quad \implies \quad \begin{cases}
                                             \rbprob{\frac{1}{k} S > p+ \varepsilon} \leq \exp(-2k \varepsilon^2) \\
  \rbprob{\frac{1}{k} S < p- \varepsilon} \leq \exp(-2k \varepsilon^2).
                                            \end{cases}
 \end{equation}
This inequality will be used to bound the three terms on the r.h.s.\,of \eqref{prob_A_in_three_pieces}.
\begin{equation*}
 \rbprob{ X^*_i > \frac{i}{k} +\varepsilon}= \rbprob{ \frac{1}{k}\sum_{\ell=1}^k \rb_ind[X_\ell \leq \frac{i}{k} +\varepsilon] < \frac{i}{k}}
\stackrel{\eqref{hoeffding}}{\leq}  \exp(-2k \varepsilon^2)
\end{equation*}
and, similarly, it holds that $\rbprob{ Y^*_j > \frac{j}{k} +\varepsilon} \leq  \exp(-2k \varepsilon^2)$. Thus it only remains to notice that
\begin{multline*}
\rbprob{A \cap \{ X^*_i < \frac{i}{k} +\varepsilon \} \cap \{ Y^*_j < \frac{j}{k} +\varepsilon \} }
\leq \\
 \rbprob{ \frac{1}{k} \sum_{l=1}^k \rb_ind [ X_l \leq \frac{i}{k} +\varepsilon, \; Y_l \leq \frac{j}{k} +\varepsilon]
 > F\left( \frac{i}{k}+\varepsilon, \frac{j}{k}+\varepsilon \right) +\varepsilon} \stackrel{\eqref{hoeffding}}{\leq}  \exp(-2k \varepsilon^2).
\end{multline*}
Thus by \eqref{prob_A_in_three_pieces} and the above applications of \eqref{hoeffding} we get \eqref{dist_of_Z_and_sample_indicator_ordered}, namely
\[ \rbprob{A} \leq 3 \exp(-2k \varepsilon^2)= 3 \exp(-2k (k^{-1/4})^2) =  3 e^{-2\sqrt{k}}. \]
The inequality
\begin{equation*}
 \forall\; i,j \in [k] \, : \; 
\rbprob{ 
\frac{1}{k} \sum_{\ell=1}^k \rb_ind [ X_\ell \leq X^*_i, Y_\ell \leq Y^*_j]
 < F\left( \frac{i}{k},\frac{j}{k}\right) -  3 k^{-1/4}} < 3 e^{-2\sqrt{k}}
\end{equation*}
can be proven analogously. Putting this and \eqref{dist_of_Z_and_sample_indicator_ordered} together we get
\eqref{dist_of_Z_and_sample_lattice}.
\end{proof}

\begin{lemma}\label{cor.amostra2}
Let $k$ be a sufficiently large positive integer. Let $n>e^{2\sqrt{k}}$ and  $\pi \in S_n$. 
Then we have 
\begin{equation}\label{sample_close_to_original_perm}
  \rbprob{ d_{\square}(\pi,\sigma(k,\pi))\leq 16k^{-1/4}} \geq 1- e^{-\sqrt{k}}.
\end{equation}

\end{lemma}
\begin{proof}
The proof uses ideas similar to the ones in the proof of Lemma \ref{permutation_and_associated_limit_permutation_subperm_densities_are_close}.

Given the permutation $\pi$, define $Z_{\pi}$ by Definition \ref{def:Z_from_perm}. The random permutations $\sigma(k,\pi)$ and $\sigma(k,Z_{\pi})$ do not have the same distribution, but, if
 $(X_1,Y_1)$, \dots, $(X_k,Y_k)$ are i.i.d.\,$[0,1]^2$-valued random variables associated with $Z_\pi$ (see Definition \ref{def:X_Y_associated_to_Z}), then 
\begin{itemize}
\item[(i)] the distribution of $\sigma(k,Z_{\pi})$ is the same as that of $S \circ R^{-1}$ (defined for $Z_{\pi}$ in \eqref{def_R_S}),
\item[(ii)]  the distribution of $\sigma(k,\pi)$ is the same as that of $S \circ R^{-1}$ under the condition that the event $B$ 
defined in \eqref{def_event_B} occurs.
\end{itemize}
It follows from \eqref{condition_can_be_dropped_if_large} that for any event $A$ we have
\[  \rbabs{\rbprob{A} -\rbcondprob{A}{B}} \leq \rbprob{B^c} \leq \frac{1}{n} \binom{k}{2} \leq e^{-2\sqrt{k}} \binom{k}{2} 
\stackrel{(*)}{\leq}
 \frac12 e^{-\sqrt{k}},
\]
where the inequality $(*)$ holds if $k$ is large enough.
If we let $A=\{d_{\square} \left( Z_\pi ,S \circ R^{-1} \right) \leq 16 k^{-1/4}\}$ then we have $\rbprob{A} \geq 1-e^{-\sqrt{k}}/2$ by
Lemma \ref{limit_perm_and_large_random_perm_are_close_in_rectangular}. Now, for $k$ large enough, \eqref{sample_close_to_original_perm} follows from
\[
 \rbprob{ d_{\square}(\pi,\sigma(k,\pi))\leq 16k^{-1/4}} =\rbcondprob{A}{B} \geq \rbprob{A}- \rbprob{B^c}\geq
 1- \frac12 e^{-\sqrt{k}} - \frac12 e^{-\sqrt{k}} \geq 1-e^{-\sqrt{k}}.
\]
\end{proof}

\section{Limits of permutation sequences}\label{section:limits_of_perm_seq}

The objective of this section is to use the machinery developed in the previous sections to prove our main results. In Subsection \ref{subsection:equivalence_of_conv_for_Z}
 we define three different types of convergence on the space of limit permutations $\mathcal{Z}$ (weak convergence, $d_{\square}$-convergence and convergence of subpermutation densities) and prove that they are all equivalent. This is then used in Subsection~\ref{subsection:poof_of_main} to prove the theorems stated in the introduction.

\subsection{Equivalence of different notions of convergence of limit permutations}\label{subsection:equivalence_of_conv_for_Z} First we prove that a limit permutation is uniquely determined by the collection of its subpermutation densities (up to the almost sure equivalence in \eqref{equivalence_of_identities}). Recall that we denote by $\mathcal{S}=\bigcup_{i=1}^\infty S_n$  the set of all finite permutations.
\begin{lemma}\label{subperm_dens_uniquely_determines_Z}
 Let $Z, \ykwhat{Z} \in \mathcal{Z}$. Denote by $(X,Y)$ and $(\ykwhat{X},\ykwhat{Y})$ the $[0,1]^2$-valued random variables associated with $Z$ and $\ykwhat{Z}$ (see Definition \ref{def:X_Y_associated_to_Z}). Then we have
 \begin{equation*}
  \; \forall \,  \tau \in \mathcal{S} \, : \, t(\tau, Z)=t(\tau,\ykwhat{Z}) \qquad
\implies \qquad (X,Y) \sim (\ykwhat{X},\ykwhat{Y}).
\end{equation*}
\end{lemma}
\begin{proof}
Denote by $F$  the joint probability
density function of $(X,Y)$  (see \eqref{distribution_fn}).
By \eqref{equivalence_of_identities} it suffices to prove that, 
 if we know $\left( t(\tau, Z)\right)_{\tau \in \mathcal{S}}$, 
 then we also
know the value of $F(x,y)$ for every $x,y \in [0,1]$. 

As a matter of fact, given $\left( t(\tau, Z)\right)_{\tau \in \mathcal{S}}$
we know the distribution of the random 
permutation $\sigma(k,Z)$ for every $k \in \N$
by \eqref{formula_subperm_density_in_Z}, and, from $\sigma(k,Z)$, we can define the random function $F_k$ as in \eqref{random_perm_dens_dist_k}. Thus we can calculate the expected value $\rbexpect{F_k(x,y)}$ 
given $\left( t(\tau, Z)\right)_{\tau \in S_k}$. As in the proof of Lemma \ref{limit_perm_and_large_random_perm_are_close_in_rectangular}, it follows
that \eqref{dist_of_Z_and_sample_distributionfunction} holds, which implies that 
for every $x,y \in [0,1]$ we have $\lim_{k \to \infty} \rbexpect{F_k(x,y)}=F(x,y)$. 

In other words, we have shown that, given $\left( t(\tau, Z)\right)_{\tau \in \mathcal{S}}$, the value of $F(x,y)$ can be recovered for every $x,y \in [0,1]$.
\end{proof}

Recall the definition of $\toinp$ from Subsection \ref{subsection:weak} and
the definition of $d_{\square}(Z_1,Z_2)$ from \eqref{rectangular_distance_def_eq}.
\begin{definition}\label{modes_of_convergence}
Let $Z,Z_1,Z_2,\ldots$ be limit permutations. 
Denote by $(X_n,Y_n)$ and $(X,Y)$ the $[0,1]^2$-valued 
random variables associated with $Z_n$ and $Z$, respectively. We say that
\begin{itemize}
 \item[1.] $Z_n \Rightarrow Z$ if $(X_n,Y_n) \toinq (X,Y)$ holds;
 \item[2.] $Z_n \toinsq Z$ if $ \lim_{n \to \infty} d_{\square}(Z_n,Z) =0$;
\item[3.] $Z_n \toint Z$ if $ \; \forall \,  \tau \in \mathcal{S} \, : \, \lim_{n \to \infty} t(\tau, Z_n)=t(\tau,Z)$.
\end{itemize}
\end{definition}

\begin{lemma}\label{different_modes_of_conv_for_Z_are_equiv}
 Let $Z,Z_1, Z_2, \ldots$ be limit permutations. The three notions of convergence of Definition
\ref{modes_of_convergence} are equivalent, that is we have
\[ Z_n \Rightarrow Z \quad \iff \quad Z_n \toinsq Z \quad \iff \quad Z_n \toint Z. \]
\end{lemma}
\begin{proof} As before, denote by $F_n$ and $F$ the respective 
joint probability distribution functions of $(X_n,Y_n)$ and $(X,Y)$, following \eqref{distribution_fn}.

We first address the equivalence of $Z_n \Rightarrow Z$ and $Z_n \toinsq Z$. On the one hand,
$Z_n \Rightarrow Z$ and $(X_n,Y_n) \toinp (X,Y)$ are equivalent by definition, and it holds that $X_n,Y_n \sim U[0,1]$. On the other hand, by \eqref{def_infty_norm} and \eqref{d_infty_d_square}, we have that $Z_n \toinsq Z$ is equivalent to
$\|F_n-F\|_\infty \to 0$. The equivalence of $Z_n \Rightarrow Z$ and $Z_n \toinsq Z$ is now a direct consequence of Lemma \ref{convergence_in_distribution_uniform_convergence_of_F}.

\medskip

Now we are going to prove that $Z_n \Rightarrow Z$ and $Z_n \toint Z$ are equivalent.

First we assume that $Z_n \Rightarrow Z$ holds. In order to deduce $Z_n \toint Z$ from this, we need to show that
$$ \; \forall \, \tau \in \mathcal{S} \, : \, \lim_{n \to \infty} t(\tau, Z_n)=t(\tau,Z).$$
Fix $k$ and $\tau \in S_k$. Let $(X_n^i,Y_n^i)$, $i=1,\dots,k,$ be i.i.d.\,with the same joint distribution as $(X_n,Y_n)$, while $(X^i,Y^i)$, $i=1,\dots,k,$ are i.i.d.\,with the same joint distribution as $(X,Y)$.

It follows from $Z_n \Rightarrow Z$ and \eqref{product_measure_weak_conv} that  
\begin{equation}\label{hosszu_eloszl_konv}
 \Big( (X^i_{n},Y^i_{n}) \Big)_{i \in [k]} \toinp \Big( (X^i,Y^i) \Big)_{i \in [k]} \qquad \textrm{ as }n \to \infty,
 \end{equation}
i.e., the sequence of $[0,1]^{2k}$-valued random variables $\Big( (X^i_{n},Y^i_{n}) \Big)_{i \in [k]}$, $n=1,2,\dots$, converges in
distribution to the $[0,1]^{2k}$-valued random variable $\Big( (X^i,Y^i) \Big)_{i \in [k]}$.

Let $A_{\tau} \subseteq [0,1]^{2k} $ denote the event that the relative order of the vertical coordinates of $(x^i,y^i)$, $i=1,\dots,k,$ is $\tau$ with respect to their ordered horizontal coordinates. Then
\begin{equation}\label{2k_dimensional_weak_convergence}
 \rbprob{ \Big( (X^i_{n},Y^i_{n}) \Big)_{i \in [k]} \in A_{\tau}}=t(\tau,Z_{n}) \quad \textrm{and}\quad
\rbprob{ \Big( (X^i,Y^i) \Big)_{i \in [k]} \in A_{\tau}}=t(\tau,Z).
 \end{equation}
In order to deduce $t(\tau, Z_n) \to t(\tau,Z)$ from \eqref{2k_dimensional_weak_convergence} we apply \eqref{portemanteau}. We only need to show that
the boundary of $A_{\tau}$ has measure $0$ with respect to the distribution of  $\Big( (X^i,Y^i) \Big)_{i \in [k]}$, which is indeed true because, on the boundary of $A_{\tau}$, either $x^i=x^j$ or $y^i =y^j$ holds for 
some $i \neq j$, but this has zero probability since both $(X^1,\dots,X^k)$ and $(Y^1,\dots,Y^k)$ are i.i.d.\,and $U[0,1]$-distributed. Thus we have proven that $Z_n \Rightarrow Z$ implies $Z_n \toint Z$.

Now we assume that $Z_n \toint Z$ holds and deduce from this that $Z_n \Rightarrow Z$ holds. By \eqref{def_weak_conv_of_prob_meas} we only need to show that for any bounded, continuous function $f:[0,1]^2 \to \RR$ we have
\begin{equation}\label{f_expect_conv}
 \rbexpect{f(X_n,Y_n)} \to \rbexpect{f(X,Y)}.
\end{equation}
We are going to prove this by contradiction: assume that there is some $f$ such that \eqref{f_expect_conv} does not hold.
 Since  $\rbabs{\rbexpect{f(X_n,Y_n)}}\leq \|f\|_\infty$, we can
choose a subsequence $\left( n(m) \right)_{m=1}^\infty$ such that 
 $\lim_{m \to \infty}\rbexpect{f(X_{n(m)},Y_{n(m)})} = a$ for some $a \neq \rbexpect{f(X,Y)}$. Now by \eqref{prokhorov} we can find a second subsequence $\left( m(\ell) \right)_{\ell=1}^\infty$ and a $[0,1]^2$-valued random variable
$(\ykwhat{X},\ykwhat{Y})$ such that if
we define $(\ykwhat{X}_\ell,\ykwhat{Y}_\ell):=(X_{n(m(\ell))},Y_{n(m(\ell))})$ then $(\ykwhat{X}_\ell,\ykwhat{Y}_\ell) \toinp (\ykwhat{X},\ykwhat{Y})$ as $\ell \to \infty$.
In particular, we have $\rbexpect{f(\ykwhat{X},\ykwhat{Y})}=a$.

 We now look at the properties of $\ykwhat{X}$ and $\ykwhat{Y}$. By \eqref{weak_conv_uniform_marginals} we know that $\ykwhat{X},\ykwhat{Y} \sim U[0,1]$. Using Lemma \ref{regular}, we define $\ykwhat{Z}$ to be the limit permutation associated with $(\ykwhat{X},\ykwhat{Y})$. By definition, we now have $\ykwhat{Z}_l \Rightarrow \ykwhat{Z}$
where $\ykwhat{Z}_\ell:=Z_{n(m(\ell))}$, and hence
$\ykwhat{Z}_\ell \toint \ykwhat{Z}$ follows from our previous argument. On the other hand, from the assumption $Z_n \toint Z$, we obtain $\ykwhat{Z}_\ell \toint Z$ .

However, from $\ykwhat{Z}_l \toint Z$, $\ykwhat{Z}_l \toint \ykwhat{Z}$ and Lemma \ref{subperm_dens_uniquely_determines_Z}, we obtain
$(X,Y) \sim (\ykwhat{X},\ykwhat{Y})$ and, in particular, $\rbexpect{f(\ykwhat{X},\ykwhat{Y})}=\rbexpect{f(X,Y)}$ which contradicts $\rbexpect{f(\ykwhat{X},\ykwhat{Y})}=a \neq \rbexpect{f(X,Y)}$. Thus we have proven that $Z_n \Rightarrow Z$ and $Z_n \toint Z$ are equivalent, which completes the proof of the lemma.
\end{proof}

\subsection{Proof of the main theorems}\label{subsection:poof_of_main} 
In this section, we shall focus on the proofs of the main results of this work, which have been stated in the introduction. We start with a brief overview of these results. Theorem~\ref{teorema_principal} establishes the correspondence between convergent permutation sequences and limit permutations, in the sense that there is a limit permutation associated with every convergent sequence and vice-versa. The fact that the limit of a permutation sequence is essentially unique is the subject of Theorem~\ref{res_princ1}, while Theorem~\ref{teorema_equiv} relates the original concept of convergence of a permutation sequence to convergence with respect to the metric $d_{\square}$.

As a consequence of Claim~\ref{claim_eventually_constant}, we shall henceforth
 restrict our attention to permutation sequences $(\sigma_n)_{n \in \mathbb{N}}$ such that $ |\sigma_n| \to \infty$.  Moreover, as a useful auxiliary fact, we point out that it follows from Definitions~\ref{def:sigma_n_to_Z} and~\ref{def:Z_from_perm} and from Lemma \ref{permutation_and_associated_limit_permutation_subperm_densities_are_close} that we have
\begin{equation}\label{convergence_for_discrete_and_smooth_are_the_same}
\sigma_n \to Z  \quad \iff \quad Z_{\sigma_n} \toint Z.
\end{equation}

\begin{proof}[Proof of Theorem \ref{teorema_principal} (i)]
Let us define $Z_n:=Z_{\sigma_n}$.
As usual, we denote by $(X_n,Y_n)$ the $[0,1]^2$-valued 
random variables associated with $Z_n$ (see Definition~\ref{def:X_Y_associated_to_Z}), which satisfy $X_{n},Y_{n} \sim U[0,1]$ for every $n$.
By~\eqref{prokhorov} there is a $[0,1]^2$-valued random variable $(X,Y)$ and a subsequence  $\left( n(m) \right)_{m=1}^\infty$
such that $(X_{n(m)},Y_{n(m)}) \toinp (X,Y)$ as $m \to \infty$.
Moreover, it follows from \eqref{weak_conv_uniform_marginals} that we have $X,Y \sim U[0,1]$. Using Lemma~\ref{regular} we define $Z$ to be the 
limit permutation associated with $(X,Y)$. With these definitions we have $Z_{n(m)} \Rightarrow Z$ as $m \to \infty$, from which
$Z_{n(m)} \toint Z$ follows by Lemma \ref{different_modes_of_conv_for_Z_are_equiv}. 
By \eqref{convergence_for_discrete_and_smooth_are_the_same} we obtain $t(\tau,\sigma_{n(m)}) \to t(\tau,Z)$ as
$m \to \infty$ for every $\tau$.
Since we have assumed that $t(\tau,\sigma_n)$ is convergent (see Definition \ref{def_conv}) as $n \to \infty$
we obtain $t(\tau,\sigma_{n}) \to t(\tau,Z)$ as $n \to \infty$, i.e.,  $\sigma_n \to Z$.
\end{proof}

\begin{proof}[Proof of Theorem \ref{teorema_principal} (ii)]
The basic idea of the proof comes from \cite{diaconis_janson}, where the authors suggest that the analogous theorem for graphs and graphons
can be proven using reverse martingales.

We jointly define the random permutations $\left(\sigma_n\right)_{n \in \N}$ in the following way: let $(X_n,Y_n)$, $n \in \N$ be i.i.d.\,$[0,1]^2$-valued random variables associated with $Z$ (see Definition \ref{def:X_Y_associated_to_Z}).
Let $\sigma_n$ be the relative order of the vertical coordinates of
 $(X_i,Y_i)$, $i=1,\dots,n$  with respect to their ordered horizontal coordinates (see Definition \ref{def:Z_random_perm}).
Note that, with this definition, $\sigma_{n}$ is almost surely a subpermutation of $\sigma_{n+1}$.

We are going to show that $\rbprob{ \sigma_n \to Z}=1$, i.e., that the random sequence of permutations $\left(\sigma_n\right)_{n \in \N}$ almost surely converges to $Z$.
 It is enough to show that, for all $\tau \in \mathcal{S}$,
we have  $\rbprob{ t(\tau,\sigma_n) \to t(\tau,Z)}=1$. Fix $\tau \in S_k$ and
let $A_{\tau} \subseteq [0,1]^{2k} $ denote the event that the relative order of the vertical coordinates of
 $(x^i,y^i)$, $i=1,\dots,k$, is $\tau$ with respect to their ordered horizontal coordinates. 
If $k \leq n$ then by \eqref{def_subperm_density_formula_cases}  the random variable $t(\tau,\sigma_n)$ 
can be expressed as
\[ t(\tau,\sigma_n)=\frac{1}{\binom{n}{k}} \sum_{(i_1,\dots,i_k) \in [n]_{<}^k}
 \rb_ind \left[ \Big( (X_{i_j},Y_{i_j}) \Big)_{j \in [k]} \in A_{\tau} \right]. \]
From this formula and \eqref{2k_dimensional_weak_convergence}, linearity of expectation leads to $\rbexpect{t(\tau,\sigma_n)}=t(\tau,Z)$.

Now  $\rbprob{ t(\tau,\sigma_n) \to t(\tau,Z)}=1$ follows from the strong law of large numbers for U-statistics
(see Section~3.4 of \cite{u_statistics}).
\end{proof}

We observe that there is an alternative, more self-contained proof of $\rbprob{ \sigma_n \to Z}=1$, with $(\sigma_n)$ as in the proof above.  In fact, we may use
the Borel-Cantelli lemma to derive $\rbprob{ d_{\square}(\sigma_n,Z) \to 0 }=1$ from Lemma
 \ref{limit_perm_and_large_random_perm_are_close_in_rectangular} and then apply Lemma \ref{different_modes_of_conv_for_Z_are_equiv}
to derive $\rbprob{ Z_{\sigma_n} \toint Z }=1$ from this and finally use \eqref{convergence_for_discrete_and_smooth_are_the_same}
to arrive at $\rbprob{ \sigma_n \to Z}=1$.

\begin{proof}[Proof of Theorem \ref{res_princ1}]
 We want to show that $\sigma_n \to Z_1$ and $\sigma_n \to Z_2$ implies that the set $\{x : Z_1(x,\cdot) \not \equiv Z_2(x,\cdot)\}$
has Lebesgue measure zero. In light of the equivalence of the statements in~\eqref{equivalence_of_identities}, this follows from Lemma \ref{subperm_dens_uniquely_determines_Z}.
\end{proof}

\begin{proof}[Proof of Theorem \ref{teorema_equiv}]
 By Claim \ref{claim_eventually_constant} we may assume that $\rbabs{\sigma_n} \to \infty$.

Let $Z_n:=Z_{\sigma_n}$. It suffices to show that $\left(\sigma_n\right)_{n \in \N}$
is a Cauchy sequence w.r.t.\ $d_{\square}$ if and only if
$Z_n \toinsq Z$ for some $Z \in \mathcal{Z}$ (i.e., the completion of $(\mathcal{S},d_{\square})$ is 
$(\mathcal{Z}/_{\sim},d_{\square})$, where $\sim$ is the equivalence relation in \eqref{equivalence_of_identities}).
Indeed, assuming that this is true, Lemma \ref{different_modes_of_conv_for_Z_are_equiv} implies that the fact that $\left(\sigma_n\right)_{n \in \N}$
is a Cauchy sequence with respect to $d_{\square}$ is equivalent to
$Z_n \toint Z$ for some limit permutation $Z$, which is equivalent to $\sigma_n \to Z$ by 
\eqref{convergence_for_discrete_and_smooth_are_the_same}, as required.

We now prove the above assertion. It is clear that, if
$Z_n \toinsq Z$ for some $Z \in \mathcal{Z}$, then $\sigma_n$ is a Cauchy sequence w.r.t.\ $d_\square$. For the converse, let $\sigma_n$ be such a Cauchy sequence and denote by $(X_n,Y_n)$ the $[0,1]^2$-valued random variables associated with $Z_n$. Denote by $F_n$ the joint probability distribution function of
$(X_n,Y_n)$. By \eqref{def_infty_norm} and \eqref{d_infty_d_square} we get that $F_n$ is a Cauchy sequence w.r.t.\ the $L_\infty$-norm.
Thus by the completeness of $L_\infty[0,1]^2$, we have $F_n \stackrel{L_\infty}{\to} F$ for some $F:[0,1]^2 \to [0,1]$. 
Now we show that $F$ itself is the joint probability distribution function of  $[0,1]^2$-valued random variable with uniform marginals.

By \eqref{prokhorov} there is an $(X,Y)$ and a 
subsequence  $\left( n(m) \right)_{m=1}^\infty$
such that $(X_{n(m)},Y_{n(m)}) \toinp (X,Y)$ as $m \to \infty$. By \eqref{weak_conv_uniform_marginals} we have $X,Y \sim U[0,1]$.
Denote by $\ykwhat{F}$ the joint probability distribution function of $(X,Y)$.
By Lemma \ref{convergence_in_distribution_uniform_convergence_of_F} we have $F_{n(m)} \stackrel{L_\infty}{\to} \ykwhat{F}$.
Comparing this to $F_n \stackrel{L_\infty}{\to} F$ we get that $F=\ykwhat{F}$. In other words, $(Z_n)$ converges to some $Z \in \mathcal{Z}$ w.r.t.\ the $L_\infty$-norm, and hence $Z_n \toinsq Z$ by \eqref{equivalence_of_identities}, as required.
\end{proof}

\noindent \emph{Acknowledgements:} The authors are indebted to an anonymous referee for valuable comments and suggestions. 

%\def\MR#1{}
%\bibliographystyle{amsplain_yk}
%\bibliography{references}

\providecommand{\bysame}{\leavevmode\hbox to3em{\hrulefill}\thinspace}
\providecommand{\MR}{\relax\ifhmode\unskip\space\fi MR }
\providecommand{\MRhref}[2]{%
  \href{http://www.ams.org/mathscinet-getitem?mr=#1}{#2}
}
\providecommand{\href}[2]{#2}
\def\MR#1{\relax}

\appendix

\section{Proof of Lemma~\ref{regular}}
In this appendix we prove Lemma~\ref{regular}. The proof is quite similar to that of Theorem~4 of Chapter~II, Section~7 of~\cite{shity}. It relies on Kolmogorov's notion of conditional expectation, about which we recall
some useful facts now.

Given a measurable space $\left(\Omega,\mathcal{F}\right)$ and an
event $A \in \mathcal{F}$ we denote by ${\rb_ind}_A=\rb_ind[A]$ the
random variable which takes value $1$ if the event $A$ occurs and $0$
if $A$ does not occur.  The following theorem states the existence and
almost sure uniqueness of conditional expectation.  For a proof, see
Section~7 of Chapter~2 of~\cite{shity}.

\begin{theorem}\label{theorem_condexp_exist_unique}
Let $\left(\Omega,\mathcal{F},  \mathbb{P} \right)$ be a probability space and let $\mathcal{G} \subseteq  \mathcal{F}$ be a $\sigma$-algebra. Consider an $\mathcal{F}$-measurable, real-valued random variable $Y$ with $\rbexpect{|Y|}<\infty$.
  \begin{enumerate}[(i)]
  \item There exists a $\mathcal{G}$-measurable, real-valued random
    variable $W$ such that $\rbexpect{|W|}<\infty$ and
    \begin{equation}\label{def_eq_condexp}
      \forall \, A \in \mathcal{G}\, : \; \rbexpect{ Y \cdot
        {\rb_ind}_A} = \rbexpect{ W \cdot {\rb_ind}_A}. 
    \end{equation}
  \item If $\widehat{W}$ is another $\mathcal{G}$-measurable,
    real-valued random variable such that
    $\rbexpect{|\widehat{W}|}<\infty$ and
    \[\forall \, A \in \mathcal{G}\, : \; \rbexpect{ Y \cdot
      {\rb_ind}_A} = \rbexpect{ \widehat{W} \cdot {\rb_ind}_A}\] then we
    have $\rbprob{W=\widehat{W}}=1$.
  \end{enumerate}
\end{theorem}

We call any $W$ that satisfies
Theorem~\ref{theorem_condexp_exist_unique}(\textit{i}) the
\textit{conditional expectation of~$Y$ with respect to~$\mathcal{G}$}
and let $\rbcondexpect{Y}{\mathcal{G}}=W$.  Note that
Theorem~\ref{theorem_condexp_exist_unique}(\textit{ii}) states that
the conditional expectation of $Y$ w.r.t.\ $\mathcal{G}$ is almost
surely uniquely defined.

If $X$ is another real-valued $\mathcal{F}$-measurable random variable
and $\sigma(X)$ is the $\sigma$-algebra generated by $X$ then we
let $\rbcondexpect{Y}{X}=\rbcondexpect{Y}{\sigma(X)}$. Since
$\rbcondexpect{Y}{X}$ is a $\sigma(X)$-measurable random variable, it
follows from Theorem~3 of Section~4 of Chapter~2 of~\cite{shity} that
there is a Borel-measurable function $\varphi: \RR \to \RR$ such that
\begin{equation}\label{measurable_sigma_X_varphi}
\rbcondexpect{Y}{X}=\varphi(X).
\end{equation}
For $A \in \mathcal{F}$, let
$\rbcondprob{A}{X}:=\rbcondexpect{\rb_ind_A}{X}$.

The proofs of the following facts about conditional expectation can be
found in Section~7 of Chapter~2 of~\cite{shity}:
\begin{equation}\label{condexp_monotone}
  Y_1 \leq Y_2 \quad \implies \quad \rbprob{  \rbcondexpect{Y_1}{\mathcal{G}} \leq \rbcondexpect{Y_2}{\mathcal{G}}}=1
\end{equation}
\begin{equation}\label{condexp_conti}
  Y_n \leq Y^*, \;  \rbexpect{Y^*}<\infty, \; Y_n \searrow Y \;\; \implies \;\;
  \rbprob{  \rbcondexpect{Y_n}{\mathcal{G}} \searrow \rbcondexpect{Y}{\mathcal{G}}}=1,
\end{equation}
where the notation $Z_n \searrow Z$ indicates that $(Z_n)$ is nonincreasing and converges to $Z$. Recall that we write~$\mathcal{B}[0,1]$ for the $\sigma$-algebra of
Borel sets of the unit interval $[0,1]$.  Denote by $\lambda$ the
Lebesgue measure on $[0,1]$.  We are now ready to give the proof of
Lemma~\ref{regular}.

\begin{proof}[Proof of Lemma \ref{regular}]
  Let $\left(\Omega,\mathcal{F}, \mathbb{P} \right)$ be the
  probability space on which our $[0,1]^2$-valued random variable
  $(X,Y)$ lives.  Recall that we assume~$X$, $Y\sim U[0,1]$ and,
  therefore, for any $B \in \mathcal{B}[0,1]$, we have $\rbprob{X \in
    B}=\rbprob{Y \in B}=\lambda(B)$.  Denote by $\tilde{\QQ}=\QQ \cap
  [0,1]$ the set of rational numbers in $[0,1]$. Thus,
  $\tilde{\QQ}$~is a countable, dense subset of $[0,1]$.  Throughout
  the proof, $\tilde{y}$, $\tilde{y}_1$ and $\tilde{y}_2$ denote
  elements of $\tilde{\QQ}$, while $x$ and $y$ denote elements of
  $[0,1]$.

  For $x \in [0,1]$ and $\tilde{y} \in \tilde{\QQ}$ define
  $\tilde{Z}(x,\tilde{y})$ to be the function for which
  \begin{equation}\label{def_eq_tilde_Z}
    \tilde{Z}(X,\tilde{y})
    \stackrel{\eqref{measurable_sigma_X_varphi}}{:=}
    \rbcondexpect{\rb_ind[Y \leq \tilde{y}]}{X}
    =\rbcondprob{Y \leq \tilde{y}}{X}. 
  \end{equation}
  Note that, given $\tilde{y}$, the function
  $\tilde{Z}(\cdot,\tilde{y})$ is only defined almost surely
  uniquely with respect to the distribution of~$X$ (i.e., the Lebesgue
  measure~$\lambda$).  We may assume $\tilde{Z}(x,1)=1$.

  For $\tilde{y} \in \tilde{\QQ}$ and $\tilde{y}_1 \leq \tilde{y}_2
  \in \tilde{\QQ}$, define $A_{\tilde{y}_1,\tilde{y}_2} \in
  \mathcal{B}[0,1] $ and $ C_{\tilde{y}} \in \mathcal{B}[0,1]$ by
  \begin{equation}
    A_{\tilde{y}_1,\tilde{y}_2}
    =\left\{ x \, : \,
      \tilde{Z}(x,\tilde{y}_1) \leq \tilde{Z}(x,\tilde{y}_2) \right\},
    \quad 
    C_{\tilde{y}}
    = \left\{ x\, : \,
      \lim_{n \to \infty} \tilde{Z}\left(x, \tilde{y}+\frac{1}{n}\right)
      = \tilde{Z}(x,\tilde{y}) \right\}.
  \end{equation}
  It follows from $\rb_ind[Y \leq \tilde{y}_1] \leq \rb_ind[Y \leq
  \tilde{y}_2]$ and~\eqref{condexp_monotone} that
  $\rbprob{\tilde{Z}(X,\tilde{y}_1) \leq \tilde{Z}(X,\tilde{y}_2)}=1$,
  which implies $\lambda( A_{\tilde{y}_1,\tilde{y}_2})=1$.  Moreover,
  it follows from $\rb_ind[Y \leq \tilde{y}+\frac{1}{n}] \searrow
  \rb_ind[Y \leq \tilde{y}]$ and~\eqref{condexp_conti} that
  $\rbprob{ \tilde{Z}(X, \tilde{y}+\frac{1}{n}) \searrow
    \tilde{Z}(X,\tilde{y})}=1$, which implies $\lambda(
  C_{\tilde{y}})=1$.

  Define $A \in \mathcal{B}[0,1]$ by
  \[ A:= \left( \bigcap_{\tilde{y}_1< \tilde{y}_2 \in \tilde{\QQ}}
    A_{\tilde{y}_1,\tilde{y}_2} \right) \cap \left( \bigcap_{\tilde{y}
      \in \tilde{\QQ}} C_{\tilde{y}} \right) .
  \]
  If $x \in A$ then $\tilde{Z}(x,\cdot)|_{\tilde{\QQ}}$ is a monotone
  increasing, right continuous function.  By the countable additivity
  of the measure $\lambda$ we have $\lambda(A)=1$.
  
  Now we construct $Z:[0,1]^2 \to [0,1]$ by defining $Z(x,1)=1$ for
  all~$x$ and, for $y \in [0,1)$, letting
  \begin{equation}\label{def_represent_Z}
    Z(x,y):= \begin{cases}
      \inf_{ y < \tilde{y}} \tilde{Z}(x, \tilde{y}) & \text{ if } x \in A \\
      y & \text{ if } x \notin A.
    \end{cases}
  \end{equation}
  Now we show that this $Z$ is a limit permutation (see Definition
  \ref{def.perm.lim}) and that it satisfies Lemma~\ref{regular}(\textit{a}). 
  Note first that it follows from~\eqref{def_eq_tilde_Z}
  and~\eqref{def_represent_Z} that~$Z$ is measurable.

It follows directly from~\eqref{def_represent_Z} that, for all $x
  \in [0,1]$, the function $Z(x,\cdot)$ is a cdf (see
  \eqref{def_cdf}), and thus $Z$ satisfies
  Definition~\ref{def.perm.lim}(\textit{a}).  For all $x \in A$ and
  $\tilde{y} \in \tilde{\QQ}$ we have
  $Z(x,\tilde{y})=\tilde{Z}(x,\tilde{y})$.

  Next we show that for any $B \in \mathcal{B}[0,1]$ and $y \in [0,1]$
  we have \eqref{reg_cond_borel_formula}:
  \begin{eqnarray*}
    \int_0^1 Z(x,y) &\rb_ind[ x \in B]& \, {d} x
    \stackrel{\lambda(A)=1}{=} \int_0^1 Z(x,y) \rb_ind[ x \in B \cap
    A] \, {d} x\\
    &\stackrel{\eqref{def_represent_Z}}{=}& 
    \int_0^1 \inf_{ y < \tilde{y}} \tilde{Z}(x, \tilde{y}) \rb_ind[ x
    \in B \cap A] \, {d} x\stackrel{(*)}{=}
    \inf_{ y < \tilde{y}} \int_0^1  \tilde{Z}(x, \tilde{y}) \rb_ind[ x \in B \cap A]   \, {d} x\\
    &=&\inf_{ y < \tilde{y}} \rbexpect{ \tilde{Z}(X, \tilde{y}) \rb_ind[
      X \in B \cap A]} \stackrel{\lambda(A)=1}{=}
    \inf_{ y < \tilde{y}} \rbexpect{ \tilde{Z}(X, \tilde{y}) \rb_ind[
      X \in B]}\\
    &\stackrel{\eqref{def_eq_tilde_Z},\eqref{def_eq_condexp} }{=}& 
    \inf_{ y < \tilde{y}} \rbexpect{ \rb_ind[Y \leq \tilde{y} ] \cdot
      \rb_ind[ X \in B] } \stackrel{(*)}{=} \rbexpect{ \rb_ind[ X \in
      B, Y \leq y ]}\\
    &=& \rbprob{X \in B,\, Y \leq y}.
  \end{eqnarray*}
  The equalities marked by $(*)$ hold because of the monotone
  convergence theorem. Our assumption $Y \sim U[0,1]$ and an
  application of~\eqref{reg_cond_borel_formula} with $B=[0,1]$ implies
  that Definition~\ref{def.perm.lim}(\textit{b}) holds, i.e., $Z$ is
  indeed a limit permutation.  This finishes the proof of Lemma
  \ref{regular}$(a)$.

  \medskip

  It only remains to prove Lemma~\ref{regular}(\textit{b}).  Assume
  that $\widehat{Z}$ is another limit permutation that
  satisfies~\eqref{reg_cond_borel_formula}.  For any given $\tilde{y}
  \in \tilde{\QQ}$, define
  \[D_{\tilde{y}}= \left\{ x \,: \,
    Z(x,\tilde{y})=\widehat{Z}(x,\tilde{y}) \right\}.
  \] 
  From the above proof it follows that both $Z(X,\tilde{y})$ and
  $\widehat{Z}(X, \tilde{y})$ satisfy the definition of
  $\rbcondexpect{\rb_ind[Y \leq \tilde{y}]}{X}$.  Thus by
  Theorem~\ref{theorem_condexp_exist_unique}(\textit{ii}) we obtain
  $\rbprob{Z(X,\tilde{y})=\widehat{Z}(X,\tilde{y}) }=1$, which implies
  $\lambda(D_{\tilde{y}})=1$. If we define $D=\bigcap_{\tilde{y}}
  D_{\tilde{y}}$ then we have $\lambda(D)=1$ by countable
  additivity. By Definition~\ref{def.perm.lim}(\textit{a}) both $Z(x,\cdot)$
  and $\widehat{Z}(x,\cdot)$ are right continuous functions for any
  $x$, and if $x \in D$ then these two functions agree on
  $\tilde{\QQ}$ and, therefore, they agree on the whole unit interval.
  This proves that $\lambda( \{ x \, : \, Z(x,\cdot) \equiv \widehat
  {Z}(x,\cdot) \})=1$, which is a reformulation
  of~\eqref{uniqueness_of_regular_condexp}, completing the proof of
  Lemma~\ref{regular}(\textit{b}).

\end{proof}

\end{document}